\documentclass[11pt,a4paper]{amsart}
\usepackage{amssymb,amsmath,layout}

\theoremstyle{plain}
\newtheorem{thrm}{Theorem}[section]
\newtheorem{lemma}[thrm]{Lemma}
\newtheorem{prop}[thrm]{Proposition}
\newtheorem{cor}[thrm]{Corollary}
\newtheorem{rmrk}[thrm]{Remark}
\newtheorem{dfn}[thrm]{Definition}

\begin{document}

\newcommand{\SL}{\mathcal L^{1,p}( D)}
\newcommand{\Lp}{L^p( Dega)}
\newcommand{\CO}{C^\infty_0( \Omega)}
\newcommand{\Rn}{\mathbb R^n}
\newcommand{\Rm}{\mathbb R^m}
\newcommand{\R}{\mathbb R}
\newcommand{\Om}{\Omega}
\newcommand{\Hn}{\mathbb H^n}
\newcommand{\aB}{\alpha B}
\newcommand{\eps}{\ve}
\newcommand{\BVX}{BV_X(\Omega)}
\newcommand{\p}{\partial}
\newcommand{\IO}{\int_\Omega}
\newcommand{\bG}{\boldsymbol{G}}
\newcommand{\bg}{\mathfrak g}
\newcommand{\bz}{\mathfrak z}
\newcommand{\bv}{\mathfrak v}
\newcommand{\Bux}{\mbox{Box}}
\newcommand{\e}{\ve}
\newcommand{\X}{\mathcal X}
\newcommand{\Y}{\mathcal Y}
\newcommand{\W}{\mathcal W}

\numberwithin{equation}{section}

\newcommand{\RN} {\mathbb{R}^N}
\newcommand{\Sob}{S^{1,p}(\Omega)}
\newcommand{\Dxk}{\frac{\partial}{\partial x_k}}
\newcommand{\Co}{C^\infty_0(\Omega)}
\newcommand{\Je}{J_\ve}
\newcommand{\eh}{\ve h}
\newcommand{\Dxi}{\frac{\partial}{\partial x_{i}}}
\newcommand{\Dyi}{\frac{\partial}{\partial y_{i}}}
\newcommand{\Dt}{\frac{\partial}{\partial t}}
\newcommand{\aBa}{(\alpha+1)B}
\newcommand{\GF}{\psi^{1+\frac{1}{2\alpha}}}
\newcommand{\GS}{\psi^{\frac12}}
\newcommand{\HFF}{\frac{\psi}{\rho}}
\newcommand{\HSS}{\frac{\psi}{\rho}}
\newcommand{\HFS}{\rho\psi^{\frac12-\frac{1}{2\alpha}}}
\newcommand{\HSF}{\frac{\psi^{\frac32+\frac{1}{2\alpha}}}{\rho}}
\newcommand{\AF}{\rho}
\newcommand{\AR}{\rho{\psi}^{\frac{1}{2}+\frac{1}{2\alpha}}}
\newcommand{\PF}{\alpha\frac{\psi}{|x|}}
\newcommand{\PS}{\alpha\frac{\psi}{\rho}}
\newcommand{\ds}{\displaystyle}
\newcommand{\Zt}{{\mathcal Z}^{t}}
\newcommand{\XPSI}{2\alpha\psi \begin{pmatrix} \frac{x}{|x|^2}\\ 0 \end{pmatrix} - 2\alpha\frac{{\psi}^2}{\rho^2}\begin{pmatrix} x \\ (\alpha +1)|x|^{-\alpha}y \end{pmatrix}}
\newcommand{\Z}{ \begin{pmatrix} x \\ (\alpha + 1)|x|^{-\alpha}y \end{pmatrix} }
\newcommand{\ZZ}{ \begin{pmatrix} xx^{t} & (\alpha + 1)|x|^{-\alpha}x y^{t}\\
     (\alpha + 1)|x|^{-\alpha}x^{t} y &   (\alpha + 1)^2  |x|^{-2\alpha}yy^{t}\end{pmatrix}}
\newcommand{\norm}[1]{\lVert#1 \rVert}
\newcommand{\ve}{\varepsilon}

\title[Gradient bounds and monotonicity, etc.]{Gradient bounds and monotonicity of the energy for some nonlinear singular diffusion equations}

\author{Agnid Banerjee}
\address{Department of Mathematics\\Purdue University \\
West Lafayette, IN 47907} \email[Agnid Banerjee]{banerja@math.purdue.edu}
\thanks{First author supported in part by the second author's
NSF Grant DMS-1001317}

\author{Nicola Garofalo}
\address{Department of Mathematics\\Purdue University \\
West Lafayette, IN 47907} \email[Nicola
Garofalo]{garofalo@math.purdue.edu}
\thanks{Second author supported in part by NSF Grant DMS-1001317}

%
%
%

\maketitle

\begin{center}
\ \textbf{Abstract}
\end{center}

\medskip

 We construct  viscosity solutions to the nonlinear evolution equation \eqref{p} below which generalizes  the motion of level sets  by mean curvature (the latter corresponds to the case $p = 1$) using the  regularization scheme as  in \cite{ES1} and \cite{SZ}. The pointwise properties of such solutions, namely the comparison principles, convergence of solutions as $p\to 1$, large-time  behavior and unweighted energy monotonicity are studied. We also prove a notable monotonicity formula for the weighted energy, thus generalizing Struwe's famous monotonicity formula for  the heat equation ($p =2$).

\medskip

\section{Introduction}

In $\Rn  \times [0,\infty)$, or in the cylinder $\Omega \times [0, \infty)$, where $\Om\subset \Rn$ is a bounded open set, we  study  the following equation 
\begin{equation}\label{me}
\text{div}(\Phi'(|Du|^2)Du)=\Phi'(|Du|^2)u_t. 
\end{equation}
Here, the function $\Phi$ is given by  
\begin{equation}\label{phi}  
\Phi(s)= \frac{2}{p} s^{\frac{p}{2}},    \quad \ \ \ \      p   \geq 1.
\end{equation}
When  $ p \to 1$, the  equation \eqref{me} becomes
\begin{equation}\label{mmc}
u_t   =  |Du|\ \text{div}\big( \frac{Du}{|Du|}\big),
\end{equation}
which  is the motion of level sets  of $u$ by mean   curvature. Most of our discussion will focus on the case $ p > 1$ of equation \eqref{me}, i.e., the equation
\begin{equation}\label{p}
|Du|^{p-2} u_t = \text{div}(|Du|^{p-2} Du). 
\end{equation}
It is worth noting that the  equation \eqref{p}  can also  be viewed as  a  generalization  of the heat  equation, which corresponds to the case $p =2$. The heat equation is also embedded in the parabolic $p$-Laplacian, 
\[
u_t  = \text{div}(|Du|^{p-2} Du),
\]
which has also been well studied, see \cite{D} and the references therein. Such equation, however, is quite different from \eqref{p} which, contrarily to the parabolic $p$-Laplacian, does not possess a divergence structure. The limiting case $p\to \infty$ of \eqref{p} is also extremely interesting
in connection with the analysis of tug-of-war games with noise in which the number of rounds is
bounded. The value functions for these games approximate a solution to the pde \eqref{p} above
when the parameter that controls the size of the possible steps goes to zero. For this, see the interesting paper \cite{MPR}.

The equation  \eqref{mmc} has  been  considered  by several authors, see for instance  \cite{ES1}-\cite{ES4}, \cite{SZ}, \cite{BG}, \cite{CGG},\cite{CW}, \cite{GGIS}, \cite{ISZ}. In \cite{ES1}-\cite{ES4} the case of   $\Rn  \times [0, \infty)$  is treated, whereas in \cite{SZ} the case of $ \Omega  \times  [0, \infty )$ is studied, with $ \Omega $ being a smooth domain  with   mean curvature bounded   from below  by a  positive constant at each point on the boundary. The existence of viscosity solutions is  proved  using  approximating   evolution equations. 
In   \cite{CGG},  equations of the  form  
\[
u_{t } +  F(Du, D^{2}u) = 0,
\]
have been considered. Unlike what was done in \cite{ES1} or \cite{SZ}, in the paper \cite{CGG} the authors proved the existence of solutions using Perron's method.
Most of the discussion in \cite{CGG}, including solvability of Cauchy problem, pertains what the authors call \emph{geometric} $F$. For the $F$ corresponding to the equation \eqref{p}, being  geometric  in the sense of \cite{CGG} is true for the case $p=1$,  but not for $p>1$. 

It should be noted that Proposition 2.2 in \cite{BG} establishes that a solution for the generalized mean curvature  flow  in the sense of \cite{ES1} is equivalent to being a solution in the sense of  \cite{CGG}. Finally, a Harnack type  approach to the evolution of surfaces  can be found in  \cite{CW}.

The present paper is organized as follows. In Section \ref{S:def} we introduce the relevant definitions of solution in the viscosity sense. In Section \ref{S:cp} we collect several comparison principles for viscosity solutions which generalize to the equation \eqref{p} those established in \cite{ES1} and \cite{CGG} when $p=1$. One notable aspect of the equation \eqref{p} is that it is invariant under the standard parabolic dilations $(x,t) \to (\lambda x,\lambda^2 t)$. Exploiting this invariance in Section \ref{S:es} we have found a notable explicit solution $G_p$ of \eqref{p}, see Proposition \ref{P:es}. By means of a variant of this solution, we have been able to establish a comparison principle for solutions of \eqref{p} which resembles the classical result of Tychonoff for the heat equation, see Theorem \ref{T:gcp}. 

In Section \ref{S:existence}, following \cite{ES1} and \cite{SZ}, we show the existence of solutions $u$ to the Cauchy and Cauchy-Dirichlet problems as limits of solutions $u^{\ve}$ of the regularized problems \eqref{e:1aprox}. We have preferred this regularization scheme since, on the one hand, it enables us to answer in the affirmative that, unlike what happens for the generalized mean curvature flow equation \eqref{mmc}, in the case $p>1$ viscosity solutions of \eqref{p} do not have finite extinction time, or finite propagation speed. On the other hand, it facilitates in the subsequent sections our study of various pointwise properties, large time behavior, monotonicity results, etc. Continuing our discussion of the plan of the paper, in Section \ref{S:conv} we show that as $p \to 1$, the corresponding solutions to \eqref{p} converge locally uniformly to the unique solution of generalized mean curvature flow \eqref{mmc}. In Section \ref{S:ltb} we study the large-time behavior  of the solutions for the Cauchy-Dirichlet problem. We first note that the $p$-energy is  non-increasing as a function of time, thus generalizing a result in \cite{SZ} for the case $p=1$. We  then  identify the  double limit  $\lim_{\ve \to 0} \lim_{t \to \infty}u^{\ve}(x,t)$ as a  solution of the  $p$-Laplace  equation subject to prescribed  boundary conditions.  In the case $ p=1$, such limit solution corresponds to the  function of least  gradient as in \cite{SZ}. Moreover,  for $1 < p \leq 2$, we show that  
\[
\underset{\ve \to 0 }{\lim}\ \underset{t \to \infty}{\lim}\ u^{\ve}(x,t) = \underset{t \to \infty}{\lim}\ \underset{\ve \to 0 }{\lim}\ u^\ve(x,t).
\]
For the case $p=1$, it was shown in \cite{SZ} that the limits in $t$ and $\ve$ do not commute, in general.

In  Section \ref{S:eemon} we first establish the monotonicity of the energy of the unique bounded viscosity solutions to the Cauchy problem constructed in the existence theorems of Section \ref{S:existence}, see Theorem \ref{T:1mon} below. It is interesting to note that the proof of such result relies crucially on the decay of solutions which is obtained by comparison with the explicit solution $G_p$ constructed in Proposition \ref{P:es} in Section \ref{S:es}. This implies, in particular, energy estimates in terms of initial datum. In the second part of Section \ref{S:eemon} we prove that, quite notably, viscosity solutions of the nonlinear singular equation \eqref{p} 
satisfy a monotonicity theorem similar to Struwe's result monotonicity theorem for the heat equation in \cite{S}, see Theorem \ref{T:struwe} below. 
 
In closing, we mention two works in which the non-geometric case ($p>1$) of the equation \eqref{me} has been studied. The existence of solution to the Cauchy problem corresponding to \eqref{me} can also be obtained by an adaptation of Perron's method, as it was done in the interesting work \cite{OS}, where a whole class of non-geometric equations was studied. Finally, the equation \eqref{p} with $p>1$ has also been studied in the interesting recent paper \cite{MPR}, where a solution to the Cauchy-Dirichlet problem is obtained by using probabilistic methods as the limit of the value functions of tug-of-war games.

\section{Preliminaries}\label{S:def}

We can formally  rewrite \eqref{me} as a non-divergence form equation as follows
\begin{equation}
\Phi'(|Du|^2)\Delta u + 2\Phi''(|Du|^2)u_{ij}u_iu_j  = \Phi'(|Du|^2)u_t,
\end{equation}
where we have let $u_i = \frac{\p u}{\p x_i}$, $u_{ij} = \frac{\p^2 u}{\p x_i \p x_j}$, etc.
After  substituting $\Phi$  as in \eqref{phi}, we obtain  
\begin{equation}\label{me2}
|Du|^{p-2}u_{t}=  |Du|^{p-4} (|Du|^{2}\delta_{ij}  + (p-2) u_iu_j)u_{ij}.
\end{equation}
We now formally proceed to cancel off the powers of  $|Du|$  from both sides of \eqref{me2},
to find 
\begin{equation}\label{me3}
u_{t} = \left(\delta_{ij} +  (p-2)\frac{u_iu_j}{|Du|^{2}}\right)u_{ij}.
\end{equation}
Motivated   by  the  above formal  calculations, following Definitions 2.1-2.3 in \cite{ES1}, we now introduce the relevant notion of solutions to \eqref{me}. Hereafter, whenever convenient we will write $z = (x,t), z_0 = (x_0,t_0)$, etc., for points in $\R^{n+1}$.
Throughout this paper, $\Om$ indicates an open set in $\Rn$ which can of course be the whole of $\Rn$, whereas $T$ indicates an extended number satisfying $0<T\le \infty$.  
 
\begin{dfn}\label{D:vsub}
A   function  $u\in C(\Om \times [0 , T))\cap L^{\infty}(\Om \times  [0, T))$ is called a \emph{viscosity subsolution} of \eqref{me}, with  $\Phi$ as  in \eqref{phi}, provided that if
\begin{equation}
u - \phi\quad  \text{has a local maximum at}\quad   z_0 \in \Om \times  (0, T)
\end{equation}
for each $\phi \in  C^{2}(\Om \times (0, T))$, then 
\begin{equation}
\begin{cases}
\phi_t \leq \left(\delta_{ij} + (p-2)  \frac{ \phi_{i}\phi_{j}}{|D\phi|^{2}}\right)\phi_{ij}\quad \text{at}\quad z_0,
\\
\text{if}\   D\phi(z_0) \not= 0,
\end{cases}
\end{equation}
and
\begin{equation}
\begin{cases}
\phi_t \leq \left(\delta_{ij} + (p-2)  a_{i}a_{j}\right)\phi_{ij}\quad \text{at} \quad z_0,\  \text{for some}\quad  a \in  \Rn\quad  \text{with}\quad  |a| \leq 1,
\\
\text{if}\   D\phi(z_0) = 0.
\end{cases}
\end{equation}
\end{dfn}

\begin{dfn}\label{D:vsup}
A   function  $u\in C(\Om \times [0 , T))\cap L^{\infty}(\Om \times  [0, T))$ is called a \emph{viscosity supersolution} of \eqref{me}, with  $\Phi$ as  in \eqref{phi}, provided that if
\begin{equation}
u - \phi\quad  \text{has a local minimum at}\quad   z_0
\end{equation}
for each $\phi \in  C^{2}(\Om \times (0, T))$, then 
\begin{equation}
\begin{cases}
\phi_t \ge \left(\delta_{ij} + (p-2)  \frac{ \phi_{i}\phi_{j}}{|D\phi|^{2}}\right)\phi_{ij}\ \text{at} \quad z_0,
\\
\text{if}\   D\phi(z_0) \not= 0,
\end{cases}
\end{equation}
and
\begin{equation}
\begin{cases}
\phi_t \ge \left(\delta_{ij} + (p-2)  a_{i}a_{j}\right)\phi_{ij}\  \text{at}\ z_0,\ \text{for some}\quad  a \in  \Rn\quad  \text{with}\quad  |a| \leq 1,
\\
\text{if}\   D\phi(z_0) = 0.
\end{cases}
\end{equation}
\end{dfn}

\begin{dfn}\label{D:vsol} A  function  $u \in C(\Om \times [0 , T))  \cap L^{\infty}(\Om\times  [0, T))$  is called a \emph{viscosity solution} of \eqref{me}  provided  it is  both a  viscosity subsolution and supersolution.
\end{dfn}

\begin{rmrk}\label{R:T}
When $T<\infty$, a viscosity   sub- or supersolution of \eqref{me} in  $\Om \times [0,T]$ is to be understood as  one in $\Om \times  [0, T + \ve) $  for some $\ve  > 0$.
\end{rmrk}

\begin{rmrk}\label{R:scales}
In a standard fashion one can verify that, if $u$ is a viscosity solution of \eqref{me}, then such is also $ku + c$, for any $k,c \in \R$. This simple, yet important property, will be repeatedly used in the present paper.
\end{rmrk} 

Following \cite{ES1}, it will be convenient to have the following equivalent definitions. 
  
\begin{dfn}[Equivalent definition]\label{D:ed}  A function $u\in C(\Om \times [0 , T)\cap L^{\infty}  (\Om \times  [0, T))$  is called  a  \emph{viscosity subsolution}   of  \eqref{me} if  whenever $z_0   \in \Om \times (0, T)$, and for some $q\in \R, \sigma\in \Rn$ and symmetric $(n+1)\times(n+1)$ matrix $R$, we have as $z\to z_0$,
\[
u(z) \leq  u(z_0)  + q(t- t_{0}) + <\sigma,x- x_{0}>  + \frac{1}{ 2} <R( z- z_{0}),z - z_{0}>  + o(|z - z_{0}|^{2}),  
\]
then 
\[
\begin{cases}
q \leq  (\delta_{ij} + ( p - 2)\frac{\sigma_{i} \sigma_{j}}{| \sigma|^{2}}) R_{ij}, \ \text{if}\  \ \ \sigma  \neq 0,  
\\
q  \leq  (\delta_{ij} + (p-2)a_{i}a_{j})R_{ij}\quad   \text{for some}\quad  a \in  \Rn, \text{with}\quad  |a| \leq 1,\ \ \ \text{if}\ \sigma = 0.
\end{cases}
\]
A \emph{viscosity supersolution} is defined similarly. Finally, $u$ is a \emph{viscosity solution} if it is at one time a viscosity sub- and supersolution.
\end{dfn}

Definitions \ref{D:vsub}, \ref{D:vsup}, \ref{D:vsol} are each easily seen to be equivalent to the corresponding case in Definition \ref{D:ed}. For this aspect we refer the reader to the seminal papers \cite{J} and \cite{I}. From Definition \ref{D:ed} one sees that a  smooth enough  viscosity solution is also a classical solution on the set where its spatial gradient does not vanish. Moreover, by adapting the argument which in Proposition 2.2 in \cite{BG} is given in the case $ p =1 $, we can conclude that, for an equation such as \eqref{me}, the notion of solution in the sense of Definition \ref{D:vsol} is equivalent to being a solution in the sense of Definition 2.1 on p. 753 in \cite{CGG}. This is the content of the  following proposition. But first, we recall the relevant definition from \cite{CGG}, adapted to the equation \eqref{me}.

\begin{dfn}\label{D:cgg}
A   function  $u\in C(\Om \times [0, T))\cap L^{\infty}(\Om \times  [0, T))$, with  $\Phi$ as  in \eqref{phi}, is called a \emph{viscosity subsolution} of \eqref{me} in the sense of \cite{CGG},  provided that if
\begin{equation}
u - \phi\quad  \text{has a local maximum at}\quad   z_0 \in \Om \times  (0, T)
\end{equation}
for every $\phi \in  C^{2}(\Om \times (0, T))$, then either
\begin{equation}
\begin{cases}
\phi_t \leq \left(\delta_{ij} + (p-2)  \frac{ \phi_{i}\phi_{j}}{|D\phi|^{2}}\right)\phi_{ij}\quad\ \ \text{at}\quad z_0,
\\
\text{if}\   D\phi(z_0) \not= 0,
\end{cases}
\end{equation}
or
\begin{equation}
\begin{cases}
\underset{|a| = 1}{\inf}\ \left\{\phi_t - (\delta_{ij}\quad  +  (p-2)  a_{i}a_{j})\phi_{ij}\right\}  \leq 0\ \ \ \text{at}\ z_0,
\\
\text{if}\   D\phi(z_0) = 0.
\end{cases}
\end{equation}
Analogous definitions for supersolution, or solution, in the sense of \cite{CGG}.
\end{dfn}

\begin{prop}\label{P:equivalence}
A bounded continuous function is a  solution in the sense of Definition \ref{D:vsol} iff it is a solution in  the sense of Definition \ref{D:cgg}.
\end{prop}

\begin{proof} 
We only look at the case of subsolution since the other  case is dealt similarly. The proof that
Definition \ref{D:cgg} $\Longrightarrow$ Definition \ref{D:vsub} is trivial, and we leave it to the reader.
We thus focus on the implication Definition \ref{D:vsub} $\Longrightarrow$ Definition \ref{D:cgg}. Suppose that for every $\phi \in  C^{2}(\Om \times (0, T))$, $u - \phi$ has a local maximum at $z_0 \in \Om \times  (0, T)$. In the case when $D\phi( z_{0}) \neq 0 $, the corresponding conditions in Definition \ref{D:vsub} and Definition \ref{D:cgg} are  seen to be the same. So we look at the case  when $D\phi (z_{0})= 0$. Without loss of generality, by replacing $\phi$ with  $\phi(x, t) + |x-x_{0}|^{4} + ( t - t_{0})^{4}$, which does not affect the spatial and time derivatives at  $z_{0}$, we can assume a strict local maximum  at $z_0$ (say, in $C_{r_0}(z_0) = \overline B_{r_0}(x_{0})\times[t_0-r_0^2,t_0+r_0^2]$). 
For $\ve>0$ we define  
\[
\xi_{\ve}(x,y,t)  =   u( x, t)  -  \frac{|x-y|^4}{\ve}  - \phi(y,t).
\]
Let $(x_{\ve},y_{\ve},t_{\ve})$ be the maximum of $\xi_{\ve}$ in the set $\tilde C_{r_0}(z_0) = \overline B_{r_0}(x_{0})\times  \overline B_{r_0}(x_{0}) \times[t_0-r_0^2,t_0+r_0^2]$. We claim that, because of the strict maximum assumption of  $u - \phi $ at  $(x_{0}, t_{0})$, the sequence $(x_{\ve},y_{\ve},t_{\ve})$ must converge to $(x_{0}, x_{0}, t_{0})$  as $\ve \to 0$.( hence for all small enough $\ve$, the maximum of $\xi_{\ve}$ is attained at an interior point of  $C_{r_0}(z_0)$)  To see this, suppose on the contrary  $(x_{\ve}, y_{\ve}, t_{\ve})$ stays away from $ (x_0, x_0, t_0)$  and
\begin{equation}\label{e:8}
\xi_{\ve}(x_{\ve}, y_{\ve}, t_{\ve}) > \xi_\ve( x_0, x_0, t_0) = u(x_0, t_0) - \phi(x_0, t_0). 
\end{equation}
Now since  $u , \phi$ are  bounded,  because of the  term $ - \frac{|x-y|^4}{\ve}$ in $\xi_{\ve}$ and the extremum condition at  $(x_{\ve}, y_{\ve}, t_{\ve})$, we  conclude that, after possibly passing to a  subsequence, we have  $(x_{\ve}, y_{\ve}, t_{\ve} ) \to (x_1, x_1, t_1) \neq (x_0, x_0, t_0)$, because of our assumption. So, from  \eqref{e:8}, we obtain 
\begin{equation}
u(x_1,t_1) -  \phi ( x_1, t_1) \geq u(x_0, t_0) - \phi(x_0, t_0),
\end{equation}
which contradicts the strict  maximum condition at $(x_0, t_0)$.

Now, if we consider the function $(y,t)\to \xi_{\ve}(x_{\ve}, y_{\ve}, t_{\ve})$, we easily see that at the point $(y_{\ve},t_{\ve})$ we have   
\begin{equation}\label{e:e}
\begin{cases}
D \phi ( y_{\ve} , t_{\ve})  =  4  \frac{ | x_{\ve} - y_{\ve}|^{2}(x_{\ve} - y_{\ve})}{ \ve},
\\
D^{2} \phi ( y_{\ve}, t_{\ve}) \geq - 4 \frac{ | x_{\ve} - y_{\ve}|^{2}}{ \ve}   -  8 \frac{( x_{\ve} - y_{\ve})^{T}( x_{\ve} -  y_{\ve})}{\ve}.
\end{cases}
\end{equation}
Two cases occur:
\begin{itemize}
\item[1)] $D\phi ( y_{\ve}, t_{\ve}) = 0$ for all $\ve$ small enough.
\item[2)] $D\phi ( y_{\ve}, t_{\ve}) \neq  0$ for a subsequence  $\ve \to 0$.
\end{itemize}
In case 1), from the first equation in \eqref{e:e} we have $ y_{\ve} = x_{\ve}$. We thus fix  $y = y_{\ve}$, and  arguing in the $x$ variable with the test function $ \frac { | x - y_{\ve}|^{4}} { \ve} + \phi (y_{\ve}, t)$, we  obtain from Definition \ref{D:vsub},
$ \phi _{t} ( y_{\ve}, t_{\ve}) \leq 0 $ (extrema condition at $(x_{\ve}=y_{\ve}, t_{\ve}))$  and so in the limit $\phi_{t}(x_{0}, t_{0}) \leq 0$ as  $ (y_{\ve}, t_{\ve}) \to (x_{0}, t_{0})$. Also  from \eqref{e:e}, we obtain  $ D^{2} \phi ( y_{\ve}, t_{\ve}) \geq 0 $ as $ y_{\ve} = x_{\ve}$,  and so likewise   $D^{2}\phi (x_{0}, t_{0}) \geq 0$. 

\medskip

Now from this information it is easily seen that for all $p>1$ we obtain at $(x_0,t_0)$  
\begin{equation}\label{e:f}
\Delta \phi   \geq  ( 2- p )\phi_{11}= (2-p) \phi_{ij}a_{i}a_{j}   \quad\ \ \text{with}\    a = e_{1}. 
\end{equation}
To see that \eqref{e:f} holds for $1 \leq p \leq 2$ we have $(2-p) \phi_{11} \leq \phi_{11} \leq \Delta \phi$. For $ p  > 2$ the right-hand side in \eqref{e:f}  is nonpositive because of nonnegativity of $D^{2} \phi$. 
So for all $p>1$ and $a=e_1$, we obtain at $(x_{0}, t_{0})$, 
$ \phi _{t}  - (\Delta \phi + ( p-2)\phi_{ij} a_{i}a_{j}) \leq  0$, which implies  Definition \ref{D:cgg}.

\medskip

 In case 2), then    
 \[
( x, t)  \to  v_{\ve}(x, t) =  u( x, t)  - | x_{\ve} - y_{\ve}|^{4} / \ve   -  \phi ( x - ( x_{\ve} - y_{\ve}), t)
\]
has  local maximum  at $( x_{\ve}, t_{\ve})$. To see this, given  any point  $(x_{\ve} + a, t_{\ve} + b) $ in the neighborhood of $(x_{\ve},t_{\ve})$  which  lies in $ C_{r_0}(z_0)$, we have
\begin{align*}
v_{\ve}(x_{\ve}+ a, t_{\ve} + b ) & =  u( x_{\ve} + a , t_{\ve} + b) - \frac{|(x_{\ve} + a)  - (y_{\ve} + a )|^4}{\ve} - \phi (y_{\ve} + a , t_{\ve} + b)
\\
& =  \xi_{\ve}( x_{\ve} + a , y_{\ve} + a , t_{\ve} + b).
\end{align*}
(adding and subtracting  $a$ in the  term $\frac{|x_{\ve}-y_{\ve}|^4}{\ve}$)
So
\begin{align*}
v_{\ve}(x_{\ve}+ a, t_{\ve} + b) & = \xi_{\ve}( x_{\ve} + a, y_{\ve} + a, t_{\ve} + b)  
\\
& \leq  \xi_{\ve} ( x_{\ve}, y_{\ve}, t_{\ve}) = u(x_{\ve},t_{\ve}) - \frac{|x_{\ve} - y_{\ve}|^{4}}{\ve} - \phi(y_{\ve}, t_{\ve})
\\
& = v_{\ve}(x_{\ve}, t_{\ve}),
\end{align*}
which justifies the local maximum of  $v_{\ve}$ at $ (x_{\ve}, t_{\ve})$.
(since maximum  of $\xi_{\ve}$ is attained at $(x_{\ve}, y_{\ve}, t_{\ve})$  in $\tilde C_{r_0}(z_0)$)
Thus, by applying Definition \ref{D:vsub} to the  test function $|x_{\ve} - y_{\ve}|^{4} / \ve  $ $+  \phi (x - (x_{\ve} - y_{\ve}),t)$, by treating it as a function of the variable $(x,t)$,  we obtain
\[
\phi_{t}  - (\delta_{ij} +  (p-2) \frac {\phi_{i}\phi_{j}}{| D\phi|^{2}}) \phi_{ij}  \leq 0,
\]
at  $(x_{\ve} - ( x_{\ve}- y_{\ve}), t_{\ve}) = (y_{\ve}, t_{\ve})$.
Let $ a^{\ve}_{i} = \phi_{i}/|D\phi| \quad at  \quad (y_{\ve}, t_{\ve})$. After  passing to a subsequence, we may assume that  $ a^{\ve}  \to  a $ with $ |a| = 1$.Therefore,  by letting ${\ve} \to 0 $,  we  obtain
\[
\phi _{t}  - (\Delta \phi    + ( p-2)\phi_{ij} a_{i}a_{j}) \leq  0
\]
at $  (x_{0}, t_{0})$, which again implies  Definition \ref{D:cgg}.

\end{proof}

The fact that equation \eqref{me} is  truly an evolution equation associated with the $p$-Laplacian is  seen from the following lemma.

\begin{lemma}\label{l:pevol}
For $p>1$ a time-independent continuous function $u(x)$ is a $p$-harmonic function if and only $U(x,t) = u(x)$ is a viscosity solution of \eqref{me}.
\end{lemma}

\begin{proof}
Suppose $u$ be a $p$-harmonic function for some $p > 1$, i.e., a $W^{1,p}_{loc}$ weak solution of the $p$-Laplacian
\[
\text{div}(|Du|^{p-2} Du) = 0.
\]
We want to show that $U(x,t) = u(x)$ is a viscosity solution of \eqref{me}. Suppose that $U -  \phi$ have a strict local maximum at  $(x_{0}, t_{0})$. Now,  in a compact neighborhood $\omega$ of $x_{0}$ in space, and thus in a compact neighborhood $K$ of $(x_{0}, t_{0})$ in space-time, there are smooth functions $u^{\ve}$ such that $u^{\ve} \to u$ uniformly, along with their first derivatives, in that neighborhood. Furthermore, in $\omega$ the functions $u^{\ve}$ solve the equation 
\[
(\delta_{ij} + (p-2)\frac{ u^{\ve}_{i}u^{\ve}_{j}}{|Du^{\ve}|^{2} + \ve^{2}})u^{\ve}_{ij} = 0,
\]
see for instance \cite{L}. Let $z_\ve = (x_{\ve},t_{\ve})\in K$ be a point at which $u^{\ve}-\phi$ has its absolute maximum. We claim that  $z_{\ve} \to z_0$ and hence  for small enough $\ve $,  $z_{\ve}$ is an interior point of $K$.  This would imply 
\[
D u^\ve(z_\ve) =  Du^{\ve}(x_\ve) = D \phi(z_\ve),\ \ \ \phi_t(z_\ve) = u^\ve_t(z_\ve) = 0,\ \ \
 D^2 u^\ve(z_\ve) = D^2 u^\ve(x_\ve) \le D^2 \phi(z_\ve).
 \]
 This implies at $z_\ve$
 \begin{equation}\label{plapphi}
(\delta_{ij} + (p-2) \frac{\phi_{i}\phi_{j}}{|D\phi|^{2} + \ve^{2}})\phi_{ij} \ge 0 = \phi_{t}.
\end{equation}
Suppose on the contrary , there exists a subsequence of $\{z_\ve\}_{\ve>0}$, which we continue to denote by $z_\ve$, which converges by compactness to a point $z_1 \neq z_0 \in K$. In fact, since $u^\ve - \phi$ attains its absolute maximum at $z_\ve \in K$, we have
\[
u^\ve(z_\ve) - \phi(z_\ve) \ge u^\ve(z_0) - \phi(z_0).
\]
Because of uniform convergence, passing to the limit in the latter inequality, we would have
\[
u(z_1) - \phi(z_1) \ge u(z_0) - \phi(z_0).
\]
This would contradict the assumption that $u-\phi$ has a strict maximum at $z_0$.
In conclusion,  $z_\ve \to z_0$  as  $\ve \to 0$, and consequently, if $D\phi(z_0)\not= 0$,
we obtain at $z_0$, after passing to the limit in the inequality \eqref{plapphi}, 
\[
\phi_{t} \leq ( \delta_{ij} + ( p-2)\frac{ \phi_{i}\phi_{j}}{|D\phi|^{2}})\phi_{ij}.
\]
If instead $D\phi(z_0) = 0$, then consider the vectors $a_\ve = \frac{D\phi(z_\ve)}{ (|D\phi(z_\ve)|^{2} + \ve^{2})^{1/2}}$. Since $|a_\ve|\le 1$ by compactness there exists a subsequence, which we continue to indicate $a_\ve$, such that $a_\ve \to a$, with $|a|\le 1$. Letting $\ve\to 0$ in \eqref{plapphi} we obtain at $z_0$
\[
\phi_{t} \leq ( \delta_{ij} + ( p-2)a_{i}a_{j})\phi_{ij}.
\]
This shows that $u$ is a subsolution. A similar argument proves that $u$ is a supersolution.

Conversely, let $u$ be a viscosity solution of  \eqref{me}. The fact that $u$ is a $p$-harmonic function is a consequence of the equivalence of  definition of viscosity and weak solution for the $p$-Laplacian established in \cite{JLM}.

\end{proof}


\section{ Comparison principles}\label{S:cp}

In this section we collect some comparison principles for viscosity solutions of \eqref{me}, both on a bounded open set $\Om\subset \Rn$, and in the whole space. Assume that $u:\Rn\times (0,\infty)$ is a continuous and bounded function. Let us  denote by  $u^{\ve}$ and $u_{\ve}$  respectively  the sup and  inf convolution of $u$, see the definitions (3.1), (3.2) and Lemma 3.1 in \cite{ES1}. It is easily seen that the conclusions of Lemma 3.1 in \cite{ES1} continue to hold in our setting. By this we mean that, in addition to the analytic properties (i)-(vi) claimed in that lemma, if $u$ is a viscosity (sub-) supersolution of \eqref{me} in $\Rn\times (0,\infty)$, then ($u^{\ve}$) $ u_{\ve}$ is a (sub-) supersolution of \eqref{me} in  $\Rn \times [\sigma(\ve),\infty)$. Using this fact, and the above stated equivalent Definition \ref{D:ed} of solution, the proof of the following Theorem \ref{T:max1} follows. It is a slight modification  of the argument  in \cite{ES1}, and it has been described in \cite{SZ} which  remains unchanged  for other $p$'s. This in particular allow us to assert uniqueness of solutions to Cauchy-Dirichlet problem. However, the result will also be used now and then at other places.

\begin{thrm}\label{T:max1}
Let $\Omega\subset \Rn$ be a bounded domain. Assume u and v are two solutions of \eqref{me} in the cylinder $\Omega_\infty = \Omega \times [0, \infty)$, with  possibly  different  initial  and  boundary data. Then, the maximum of  $|u-v|$ is achieved  on the parabolic boundary $\p \Om_\infty = \p \Om\times (0,\infty) \cup \Om\times\{0\}$.
\end{thrm}

Furthermore, since  our notion of sub- and supersolution is the same as  that of \cite{CGG}, see Proposition \ref{P:equivalence} above, we also have the usual parabolic  comparison  principle (see \cite{CGG}, Theorem 4.1 or \cite{GGIS}). 

\begin{thrm}\label{T:max2}
Let $\Omega\subset \Rn$ be a bounded domain, and let $u$ and $v$ respectively be sub- and supersolution of \eqref{me} in $\Omega \times [0, T]$. If $u \leq v$ on the parabolic boundary, then $ u \leq v$ on $\overline{\Omega} \times [0, T]$.
\end{thrm} 

Using again Proposition \ref{P:equivalence}, we obtain a comparison theorem when $\Omega = \Rn$, see Theorem 2.1 in \cite{GGIS}. This implies uniqueness in the  Cauchy problem (the reader should keep in mind that, for the notions adopted in this paper, all solutions are a priori assumed to be bounded).

\begin{thrm}\label{T:max3}
Let $u$ and $v$ respectively be  sub- and supersolution of \eqref{me} in $\Rn \times [0, \infty)$ such that  $u (x, 0) \leq v(x, 0 )$. Let $u(\cdot, 0)$ and $v(\cdot, 0)$ be uniformly continuous. Then,  $u \leq v $  in $\Rn \times [0, \infty)$.
\end{thrm}

We note explicitly that the assumption that the initial datum be uniformly continuous in Theorems \ref{T:max3} and \ref{T:max4} is needed to guarantee that the  condition (A2) in Theorem 2.1 of \cite{GGIS} be satisfied.
After a  minor modification in the proof of Theorem \ref{T:max3}, we can assert the following.

\begin{thrm}\label{T:max4}
Let $u$ and $v$ be two solutions of \eqref{me} in $\Rn \times [0, \infty)$. Let  $u(\cdot, 0)$ and $v(\cdot, 0)$ be uniformly  continuous. Then,  
\[
\underset{\Rn \times [0, \infty)}{\sup}\ | u - v| = \underset{\Rn}{\sup}\ |u(\cdot, 0) - v(\cdot, 0)|.
\]
\end{thrm}
\begin{proof}
We argue by contradiction. Suppose the conclusion of the theorem be not true. Then, without loss of generality, we can assume that there exist $T > 0$ such that
\[
\underset{\Rn \times [0, T]}{\sup}\ (u-v)  > \underset{\Rn}{\sup}\ [u(\cdot,0)-v(\cdot,0)]^{+} > 0.
\]
Otherwise, if  $\underset{\Rn}{\sup}\ [u(\cdot,0)-v(\cdot,0)]^{+} = 0$, and Theorem \ref{T:max3} would  imply  $\underset{\Rn \times [0, T]}{\sup}\ (u-v) \leq 0$.
Following \cite{GGIS}, we set $\omega(x, y, t) = u(x, t) - v(y, t)$, $B(x, y, t)=  \delta(|x|^{2} + |y|^{2}) + \frac{\gamma}{T-t}$, ($B$ plays the role of barrier), and $\xi(x, y, t)= \frac{|x-y|^{4}}{\ve} + B(x, y, t)$. We then consider 
\[
\Phi(x, y , t) = \omega( x, y , t) - \xi(x, y, t).
\]
By the contradiction assumption we have 
\begin{align*}
\alpha & = \underset{\theta \to 0}{\lim}\ \sup\ \{\omega(x, y, t) \mid |x-y| < \theta,\ 0\le t \leq T\} \geq \underset{0\le t \leq T}{\sup}\ \omega(x,x,t)
\\
& = \beta > \underset{\Rn}{\sup}\ [u(\cdot,0)-v(\cdot,0)]^{+} =\beta_{1}.
\end{align*}
Similarly to Proposition 2.4 in \cite{GGIS}, we can find sufficiently small $\delta_{0}, \gamma_{0}$ such that 
\[
\sup\ \Phi( x,y, t) >   \beta_{1},
\]
for all $\delta < \delta_{0}$, $ \gamma < \gamma_{0}$,
and because of the barrier function $B$, we obtain as in Proposition 2.5 in \cite{GGIS} that the sup is attained  at some $(x_1, y_1, t_1)$ with  $t_{1} < T $. Now, by making use of the condition that the initial datum is uniformly continuous, which guarantees the validity of (A2) in Theorem 2.1 in \cite{GGIS}, as in Proposition 2.6 in \cite{GGIS} we can find $\ve_{0}$ small enough such that  for all $\ve < \ve_{0}$,  sup $\Phi$ is attained at  $(x_1, y_1, t_1)$ with $t_{1} > 0$ (i.e., at an interior point). The rest of the proof is now identical to that in \cite{GGIS}.

\end{proof}

In closing, we mention the following additional properties of solutions which can be established as in \cite{ES1}:
\begin{itemize}
\item[1.]  Assume  $u_{k}$  is  a  viscosity solution  of \eqref{me} for  $k = 1, 2,...$,  and  that $ u_{k} \to  u$ locally  uniformly   on $\Rn \times [0,\infty)$. Then, $u$ is a  viscosity solution.  An analogous assertion holds for subsolutions and  supersolutions.
\item[2.] (Convexity preserving property, see Theorem 3.1 in \cite{GGIS}) If $u$ is a (bounded) solution to the Cauchy problem and  if the initial datum $g$ is concave/convex and  globally Lipschitz, then $u(\cdot, t)$ is concave/convex for all $t$. This property continues to be true if instead of boundedness, linear growth in $x$ is allowed.
\item[3.] If $p=1$ and $\Phi$ is  any smooth function, then $\Phi(u)$ is a viscosity solution if such is $u$, see \cite{ES1}. This property is no longer true in general when $p>1$. For instance, it is violated by the standard heat equation ($p=2$).
\end{itemize}



\section{ An explicit solution for the case $p > 1$}\label{S:es}

In this section we exploit the scaling properties of \eqref{me} to construct an interesting explicit global viscosity solution of \eqref{me}, which we later use to establish a Tichonoff type maximum principle for the equation \eqref{me}, see Theorem \ref{T:gcp} below. In fact, such explicit solution will also be used in a crucial way in the proof of the monotonicity results in Theorem \ref{T:1mon}  and Theorem \ref{T:struwe} below. Here is the relevant result.

\begin{prop}\label{P:es}
For any $p>1$ the function
\[
G_p(x,t) = t^{- \frac{n+p-2}{2(p -1)}} \exp \left({-\frac{|x|^2}{4(p-1)t}}\right),
\]
is a classical solution of \eqref{me} in $(\Rn -\{0\})\times(0,\infty)$, and a viscosity solution in $\Rn\times(\ve,\infty)$ for all $ \ve > 0 $.
\end{prop}

\begin{proof}
Suppose that $u$ be a  solution of \eqref{me}, then it is easily seen that $v(x, t) = u(\lambda x, \lambda^{2} t)$ is also a solution.
This suggests that we look for a solution of \eqref{me} in the form 
\[
u(x, t) = t^{-\alpha} g(|x|^{2}/4t), 
\]
where $g$ is a suitable function on the line. Proceeding formally, we work with the normalized equation \eqref{me3},
\[
u_{t} = \left(\delta_{ij} +  (p-2)\frac{u_iu_j}{|Du|^{2}}\right)u_{ij},
\]
 instead of \eqref{me}. To determine $g$ we compute
\[
u_{t} = - \frac{\alpha} {t^{\alpha + 1}} g - \frac{|x|^{2}}{4t^{\alpha + 2}} g'.
\]
Similarly, 
\[
D_j u = \frac{1}{2t^{\alpha + 1}} g' x_j,\ \ \ \   D_{ij} u = \frac{1}{2 t^{\alpha + 1}} g' \delta_{ij} + \frac{1}{4 t^{\alpha + 2}} g'' x_{i}x_{j}.
\]
Imposing that $u$ be a (classical) solution of \eqref{me3}, we find
\[
- \frac{\alpha}{ t^{\alpha+1}} g -  \frac{|x|^{2} }{4 t^{\alpha + 2}} g'  = \left\{\frac{n + p - 2}{2} g' + (p-1) \frac{|x|^{2}}{4t} g''\right\}  \frac{1}{t^{\alpha + 1}}.
\]
Cancelling off the powers of  $t$, and  letting $s = \frac{|x|^{2}}{4t}$,  we find
\[
g'' + \frac{1}{p-1} g' + \frac{n + p - 2}{2(p-1) s} g' + \frac{\alpha}{(p-1) s} g = 0.
\]
At this point, we choose $\alpha  =  \frac{n + p - 2}{2(p - 1)}$, obtaining
\[
\left(s^{\alpha} g' + \frac{1}{p-1} s^ {\alpha} g\right)' =  0,
\]
which implies 
\[
s^{\alpha} g' + \frac{1}{p -1} s^{\alpha} g =\ \text{const}.   
\]
At this point we easily see that the choice $g(s) = \exp(-\frac{s}{p-1})$ produces a solution of the latter ode corresponding to the choice $c = 0$ of the constant in the right-hand side. Since the function 
\[
u(x,t) = G_p(x,t) = t^{- \frac{n+p-2}{2(p -1)}} \exp \left({-\frac{|x|^2}{4(p-1)t}}\right),
\]
belongs to $C^\infty(\Rn\times (0,\infty))$, and its spatial gradient
vanishes only on the half-line $\{0\}\times (0,\infty)$, we conclude that $G_p$
is a classical solution of \eqref{me}  in $(\Rn - \{0\}) \times  (0,\infty)$ for every $p>1$. Now, consider $t>0$, and let $z_k = (x_{k},t) \to z = (0,t)$, with $x_k\not= 0$ for every $k\in \mathbb N$. Define $a_{k} =  \frac{Du(z_k)}{|Du(z_k)|}$. After possibly passing to a subsequence, we may assume that  $a_{k} \to a$, with  $|a|=1$. 
Since  at each $z_k$, \eqref{me3} is satisfied classically, we have
\[
u_{t}(z_k) =  (\delta_{ij} +( p-2)a_{k,i} a_{k,j})u_{ij}(z_k).
\]
Therefore, passing to the limit we obtain 
\[
u_{t}(z) =  (\delta_{ij} +( p-2)a_{i} a_{j})u_{ij}(z).
\]
Thus, $u$ is a viscosity solution in $\Rn \times (\ve,\infty)$  for all $\ve > 0 $.

\end{proof}

We note explicitly that, when  $p = 2$, $G_p$ is just a multiple of the classical heat kernel. In what follows we use a variant of the  special solution $G_p$ as a comparison function to obtain the following result for solutions of \eqref{me} which generalizes the classical global maximum principle of Tichonoff type for the heat equation. We emphasize that in the next result we assume that $u$ be a solution in the sense of Definition \ref{D:vsol}, but we do not a priori impose the condition that $u \in L^{\infty}(\Rn\times (0,\infty))$.

\begin{thrm}\label{T:gcp}
Let $u\in C(\Rn\times [0,\infty))$ be such that $u$ is a viscosity solution of \eqref{me} in $\Om\times [0,T]$ for every bounded open set $\Om\subset \Rn$ and every $T>0$. Assume further that for some $A, a>0$ the following growth estimate is satisfied:  
\[
u(x,t) \leq A e^{a|x|^{2}},\ \ \ \ \text{for every}\ (x,t)\in \Rn\times [0,\infty).
\]
Then,  
\[
\underset{\Rn \times [0, T]}{\sup}\ u = \underset{\Rn}{\sup}\ u(x,0).
\]  
\end{thrm}

\begin{proof}
If $\underset{\Rn}{\sup}\ u(x,0) = + \infty$, there is nothing to prove, so we will assume without loss of generality that 
\[
K = \underset{\Rn}{\sup}\ u(x,0) < + \infty.
\]
Let  us first assume that  $4a( p -1 )T < 1$. This implies that $4a( p-1)(T + \ve) < 1$, for some  $\ve > 0$. Now,  fix $y \in \Rn $ and  $\mu > 0$. 
Let 
\[
v( x, t) =  K  + \frac{\mu}{( T+\ve - t)^{\frac{n+p-2}{2p-2}}} e^{\frac{|x-y|^{2}}{4(p-1)(T+ \ve - t )}}.
\]
As already observed in Proposition \ref{P:es}, the equation \eqref{p} is invariant under addition and multiplication by a constant. One can thus easily check that $v$ is a solution of  \eqref{me} in $\Om_T = \Om \times [0, T]$ for every bounded open set $\Om\subset \Rn$.
Now  let  $\Om = B(y, r)$. So $u$ and $v$ are solutions of \eqref{me} in  $\Om_{T} $, and 
$v( x, 0) \geq K  \geq u(x, 0)$ for every $x\in \overline \Om$.
On $ \partial \Om \times [0, T] $,  i.e  $ | x-y | = r$, we have
\[
v( x, t)=  K +  \frac{\mu}{( T+\ve - t)^{\frac{n+p-2}{2p-2}}} e^{\frac{r^{2}}{4(p-1)(T+ \ve - t )}}  \geq  K  +  \frac{\mu}{( T+\ve )^{\frac{n+p-2}{2p-2}}} e^{\frac{r^{2}}{4(p-1)(T+ \ve )}}.
\]
Now, we can write $\frac{1}{4(p-1)(T+ \ve) } = a  + \gamma$, for some $\gamma > 0$. Thus,  we  find that for $r$ large enough (since $ y$ is fixed)
\[
v( x, t) \geq K + \mu ( 4(p-1) ( a + \gamma ))^{\frac{n+ p-2}{2p-2}} e^{(a+ \gamma)r^{2}} > A e^{a( r+ |y|)^{2} }> u ( x, t),
\]
for every $(x,t)\in \partial \Om \times [0,T]$.
Therefore, by the comparison principle  Theorem \ref{T:max2} we conclude that
\[
u(x, t) \leq  \underset{\Rn}{\sup}\ u(x,0) + \frac{\mu}{( T+\ve - t)^{\frac{n+p-2}{2p-2}}} e^{\frac{|x-y|^{2}}{4(p-1)(T+ \ve - t )}}
\]
for all $(x, t) \in \overline \Om \times [0, T]$. Letting  $\mu \to  0$, we reach the conclusion 
\[
u(x, t) \leq  \underset{\Rn}{\sup}\ u(x,0),\ \ \ \ \text{in}\  \overline B(y,r) \times [0, T].
\]
By the arbitrariness of $r>0$ we conclude that the above inequality holds in $\Rn\times [0,T]$, provided that 
$4a(p-1)T < 1$. The dsired conclusion now follows by repeatedly  applying this result  to the intervals $[0,  T_{1}], [T_{1}, 2T_{1}]$ and so on, where, say,  $ T_{1}  = 1/8a(p-1)$.

\end{proof}

By combining Theorem \ref{T:gcp} with  Theorem \ref{T:max3},  we obtain the  following result.

\begin{prop}[Uniqueness for Cauchy problem]\label{P:uniq}
In the class of functions $u\in C(\Rn\times [0,\infty))$ satisfying for some $A, a > 0$ the growth estimate
\[
|u(x,t)| \leq A e^{a |x|^{2}},\ \ \ \ \ (x,t)\in \Rn\times[0,\infty),
\]
there exists at most one solution to the Cauchy problem for \eqref{me} with a bounded  uniformly continuous initial datum.  
\end{prop}

\section{Existence of   solutions}\label{S:existence}

In this section using  the method of regularization as in \cite{ES1} we prove the existence of solutions to the Cauchy problem and the Cauchy-Dirichlet problem. An alternate approach to the existence of solutions to the Cauchy problem is based on the adaptation of the method of Perron described in Theorem 4.9 of \cite{OS}. However, we have preferred the method of regularization since it facilitates the study, in the subsequent sections, of questions of convergence, large time behavior, gradient bounds and monotonicity.

For every $i,j=1,...,n$, and $\sigma\in \Rn\setminus\{0\}$ consider the matrix associated with \eqref{p} 
\begin{equation}\label{matrix}
a_{ij}(\sigma)=\delta_{ij} + (p-2) \frac{\sigma_i \sigma_j}{ |\sigma|^{2}}.
\end{equation}
In the sequel for a given $\ve>0$ we consider the regularized matrix 
\[
a^{\ve}_{ij}(\sigma)=  \delta_{ij} + (p-2) \frac{\sigma_i \sigma_j}{\ve^{2} + |\sigma|^{2}},\ \ \ \ i, j = 1,...,n.
\]
It is easily seen that for every $\sigma\in \Rn$ and every $\xi\in \Rn$ the following uniform ellipticity condition is satisfied, independently of $\ve>0$,
\begin{equation}\label{ue}
\min\{1,p-1\}\ |\xi|^2 \le a^{\ve}_{ij}(\sigma)\xi_i\xi_j  \leq \max\{1,p-1\}\ |\xi|^2.
\end{equation}
 
\subsection{Approximations} 

We consider solutions of  
\begin{equation}\label{e:1aprox}
\begin{cases}
u_t =  a^{\ve}_{ij}(Du)u_{ij}\ \ \ \ \ \text{in}\ \Rn \times [0,\infty)
\\
u(x,0) =   g(x),\ \ \ \ \  \ x\in \Rn,
\end{cases}
\end{equation}
where  $g$ is smooth and, for  some  $S > 0$, $g$ is constant  for $|x| \geq  S$. To be precise, we should indicate with $u = u^{\ve}$ the solution of \eqref{e:1aprox}, and when this will be necessary we will do so (although the notation $u^\ve$ is the same that indicates the sup convolutions mentioned in the opening of Section \ref{S:cp}, hereafter, there will be no occasion for confusion).     
In the case $ p =1$ studied in \cite{ES1}, the uniform ellipticity breaks down, as can be seen from \eqref{ue} above. Because of this, the authors had to consider the further regularization $a^{\ve,\eta}_{ij}(\sigma)  =  a^{\ve}_{ij}(\sigma)+\eta \delta_{ij}$. This is not needed for the situation $p >1$ studied in this paper.
In a classical way, for $p>1$, we obtain smooth  bounded  solutions $u^{\ve}$ of \eqref{e:1aprox}. 
Using the classical theory, we have
\[
 ||Du^{\ve}||_{L^{\infty}(\Rn \times [0, \infty))} =  ||Dg||_{L^\infty(\Rn)},
 \]
and   
\[
||u^{\ve}_{t}||_{L^{\infty}(\Rn \times [0, \infty))} \leq C  ||D^{2}g||_{L^\infty(\Rn)}, 
\]
where $C$ does not depend on $\ve$. For instance, the second inequality above can be justified differentiating the equation with respect to $t$. Then, by the maximum principle we obtain
 \[
 ||u^{\ve}_{t}||_{L^\infty(\Rn \times [0, \infty))} = ||u^{\ve}_{t}(., 0)||_{L^\infty(\Rn)} = ||a^{\ve}_{ij} (Dg)D_{ij} g||_{L^\infty(\Rn)} \leq (p-1) ||D^2 g||_{L^\infty(\Rn)}.
 \]
Also from the uniform ellipticity \eqref{ue}, the smoothness of the coefficients and classical estimates, we obtain  $L^{\infty}$  bounds on all higher-order derivatives, possibly depending on $\ve$ and $g$. An outline goes as follows. 

We  take a sequence of smooth domains $\Omega ^{N} \nearrow  \Rn$. Given any $T>0$, we consider the finite cylinders $\Om^N_T = \Om^N \times (0,T)$, and indicate with $\p_p \Om^N_T = (\p \Om^N \times (0,T))\cup (\Om^N \times \{0\})$ its parabolic boundary. For each $n\in \mathbb N$,   and $\ve>0$, we solve the Cauchy-Dirichlet  problem  
\[
\begin{cases}
u^{\ve,N}_{t}  =  a^{\ve}(Du^{\ve,N}) u^{\ve,N}_{ij},\ \ \ \ \text{in}\ \Om^N_T,
\\
u^{\ve,N} = g\ \ \ \text{on}\ \p_p \Om^N_T\ \  (\text{one should keep in mind that}\ g(x,t) = g(x)).
\end{cases}
\]
Since  $g$ is  constant  outside  a   compact  set, for sufficiently large $N$ the compatibility conditions are obeyed. Therefore, the  existence of unique solutions in the class  $H^{{2+\alpha},  {1+ \alpha/2}}(  \overline{ \Omega^{N} } \times  [0, T])$  is guaranteed by Theorem 4.1, on p. 558 in \cite{LU}. Because of the parabolic comparison principle, the solutions  obtained  for various $T$'s agree on the intersection of the corresponding cylinders, and thus we have  a solution in $\Omega^{N} \times  [0, \infty)$. By the maximum principle,  
\[
||u^{\ve,N}||_{L^\infty(\Om_N\times (0,\infty))}  \leq ||g||_{L^\infty(\Rn)}.
\]
Now  using  a  barrier  argument as in \cite{SZ}, p. 586, we have for a constant $C$, which depends on $p-1$,
\[
||u^{\ve,N}_{t}|| _{\infty} \leq  C ||D^{2} g||.
\]
The local Schauder theory gives the existence of higher derivatives. See, for instance, Theorem 10 on p. 72 in \cite{F}. This allows differentiation with respect to the space variables $x_{k}$. Now by Theorem 3.4 on p. 554 in \cite{LU}, we  have for  $ N >> m $
\[
||D u^{\ve,N}||_{\Omega^{m}_{T}} \leq c(m,g).
\]
Consequently, by Theorem 1.1 on p. 517 in \cite{LU}, we have the following H\"older norm estimate  
\[
< Du^{\ve,N}>_{\alpha} \leq C(m, g), \ \ \ \ \text{for}\ N >> m. 
\]
From such estimate, by applying the Schauder estimates as on p. 121 in \cite{F}, or on p. 352 in \cite{LU}, we obtain for $N  > > m$  and any $T>0$, 
\[
|| u^{\ve,N}||^{\Omega^{m}_{T}} _{2+ \alpha}   \leq   C_{1}(m,g,T). 
\]
By  a standard diagonal process,  we now obtain a subsequence that  converges with its first and second derivatives uniformly in compact subsets of $\Rn \times [0,\infty)$ to  some bounded $u^\ve$ which  solves  \eqref{e:1aprox}.
Now  given any   compact   set  $K$  in  $\Rn \times  [0,\infty)$, by applying  Theorem  3.4, page 554 in \cite{LU}, we deduce  a  bound on $||Du^\ve||_{L^\infty(K)}$  independent of  $K$, and thus  an uniform  bound  on  $Du^\ve$  in  $\Rn \times  [0, \infty)$. Therefore, for this $u^\ve$, we have  
\begin{equation}\label{e:e10}
\ ||u^\ve||_\infty =||g||_\infty, \ \ \ ||u^{\ve}_{t}||_\infty \leq (p-1) ||D^{2}g||_\infty,\   \ \ ||Du||_\infty \leq C(\ve,g). 
\end{equation}
At this point, the bound $C$ for $||Du||_\infty$ as above might possibly depend on $\ve$ (see Theorem 3.4 on p. 554 in \cite{LU} for the dependence of the constant),
and we also have $Du^\ve(\cdot,t) \to Dg$  as  $t \to 0$ since this happens for $Du^{\ve, N}$.
Theorem 1.1, p. 517 in \cite{LU} guarantees, in any compact set $K$, a uniform bound on the H\"older norm of $Du^\ve$ independent of $K$. As before, the local Schauder theory as in Theorem 10, p. 72 in \cite{F}, gives the existence of higher derivatives.
Consequently, by appealing to the linear theory, we can conclude (for instance, by Theorem  5.1 on p.320 in \cite{LU}), the smoothness of $u^\ve$, and the bounds on the higher derivatives of $u^\ve$. By differentiating the equation with respect to $x_{k}$, we obtain 
\[
u^\ve_{kt} =  a^{\ve}_{ij}(Du^\ve) u^\ve_{kij}  + a^{\ve}_{ij, \sigma_{\ell}} (Du^\ve) u^\ve_{\ell k}u^\ve_{ij}.
\]
Since $Du^\ve$ is  bounded, by the maximum  principle, we have $||Du^\ve||_{\infty} = ||Dg||_{\infty}$  (and thus  we  get   a   constant  $C$ independent of  $\ve$).
Thus, we can  finally  assert  
\[
||u^\ve, Du^\ve, u^\ve_{t} ||_\infty \leq  C(g),
\]
where $C$ is independent of $\ve>0$.

\subsection{Passage to the limits}\label{SS:passage}

We  have the following  existence theorem analogous to Theorem 4.2 in \cite{ES1}.

\begin{thrm}\label{T:existence}  Given a continuous function $g$, which for some $S > 0$ is constant  for $|x| \ge S$,  there  exists  a  unique viscosity solution $u$ of \eqref{me} such that $u = g$ on $\Rn \times \{0\}$.   
\end{thrm}

\begin{proof}
Given   a  smooth $g$  constant  outside a  compact set by arguing as above we have solutions to the approximating problems. From  uniform  bounds on the  derivatives,  we   extract a subsequence $u^{\ve_k}$ so that  $u^{\ve_k} \to u$,  locally uniformly in $\Rn \times  [0, \infty)$, to  some   bounded  Lipschitz function  $u$  which inherits the Lipschitz constant  from the bounds on the second derivatives of  $g$. The proof  that $u$  is  a viscosity solution  is  now   identical  to that  given  in \cite{ES1}. We now analyze solvability for a  continuous $g$. We take smooth $g_{k}'s $ converging uniformly to $g$. Let the corresponding solutions be $u_{k}$. From Theorem \ref{T:max4} we have  $u_{k} \to u$ uniformly. Since $u$ is a uniform limit of solutions, it is itself a solution, and takes the initial values $g$. Uniqueness follows from the comparison principle.

\end{proof}

\begin{rmrk}\label{R:one} Unlike  the case $p=1$, it cannot  be asserted that  the  solutions are  constant  outside  large sets  in space-time since, even for the  case $ p =2$, the  bounded  solutions for a  compactly  supported initial datum do not obey this  property. The  auxiliary function  $\psi$ as  in \cite{ES1} (page 657) can be seen to be very specific  for the  case $  p=1$. In general   for any $p > 1$,  say $g$ is non-negative and compactly supported, then $u^{\ve}$  as  above are solutions of a uniformly parabolic partial differential equation in nondivergence form, with eigenvalues controlled from above and below by  $\max\{1,p-1\}$ and $\min\{1,p-1\}$, respectively. Thus, the Harnack inequality holds for $u^{\ve}$ (see \cite{KS}) with constant independent of $\ve$, and therefore in the limit the Harnack inequality is satisfied by $u$. This in particular rules out  finite extinction time and  finite  propagation speed when $p > 1$.
\end{rmrk}

\begin{rmrk}\label{R:two} We cannot conclude that  $u$  is smooth for a  smooth datum as the  bounds on higher derivatives  of  $u^{\ve}$ obtained from the parabolic theory depend on $\ve$.
\end{rmrk}


\subsection{The case of bounded domains}\label{SS:bded}


In what follows we consider a $C^2$ convex domain $\Omega\subset \Rn$, with mean curvature bounded from below by a  positive constant at each point on the boundary. This geometric assumption was introduced in \cite{SZ}. We intend to establish the following result.

\begin{thrm}\label{T:existencebdd}
Given $g\in C(\overline \Om)$. For any $p>1$ there exists a unique viscosity solution of the Cauchy-Dirichlet problem 
\begin{equation}\label{cd}
\begin{cases}
\emph{div}(|Du|^{p-2} Du) = |Du|^{p-2} u_t\ \ \ \ \ \text{in}\ \Om \times (0,\infty),
\\
u(x,t) = g(x),  \ (x,t) \in \partial_p(\Om\times(0,\infty)).
\end{cases}
\end{equation}
\end{thrm}

\begin{proof}
With $a_{ij}(\sigma)$ as in \eqref{matrix}, our objective is finding a viscosity solution of  
\begin{equation}\label{e:bdd}
\begin{cases}
u_t =  a_{ij}(Du)u_{ij}\ \ \ \ \ \text{in}\ \Om \times (0,\infty),
\\
u(x,0) =   g(x),\ \ \ \ \   \ x\in \Om,  \ \ u(x, t) = g(x),  \ (x,t) \in \partial \Om\times(0,\infty).
\end{cases}
\end{equation}
As  in \cite{SZ}, $u$  can be obtained as the  limit of  the  solutions of  the following approximating problems 
\begin{equation}\label{e:1aprox0}
\begin{cases}
u_t =  a^{\ve}_{ij}(Du)u_{ij}\ \ \ \ \ \text{in}\ \Om \times (0,\infty)
\\
u(x,0) =   g(x),\ \ \ \ \   \ x\in \Om,  \ \ u(x, t) = g(x),  \ (x,t) \in \partial \Om\times(0,\infty),
\end{cases}
\end{equation}
where we  first assume $g$ be smooth.
For a  given $\ve$, the existence  of a classical solution $u^{\ve}$ which is H\"older continuous in $\overline{\Omega} \times [0, \infty)$  is guaranteed by \cite{LU} (Theorem 4.2  page 559). We note in passing that since, as before, we are dealing with the case $p > 1$, we do not need the additional regularization by $\sigma \delta_{ij}$ which is necessary in the case $p=1$.
As in \cite{SZ}, the uniform Lipschitz bounds are obtained by means of a  barrier method. In view of our assumption that the mean curvature of $\partial \Om$ be bounded from below by a positive constant, in order to bound the spatial  difference quotients a barrier function of the form  $\lambda d( x, \partial \Omega)$ will do. Similarly, the time difference quotients can be  bounded by using a barrier of the form $\alpha t$, where  $\alpha$ depends only on the $C^{2}$ norm of $g$. Once we have  uniform Lipschitz bounds, we can pass to the limit, and obtain a Lipschitz  viscosity solution $u$ of \eqref{p} in $\Omega  \times [0, \infty)$ when the Cauchy-Dirichlet datum $g$ is smooth. In the case of a continuous datum, we can approximate $g$ uniformly  by smooth functions $g_{k}$ and denote  by $u_{k}$ the corresponding solutions. By applying Theorem \ref{T:max1}, we can conclude that $ u_{k} \to u$ uniformly, with $u$ taking up the boundary value $g$. Since  $u$ is a  uniform limit of $u_{k}$, also  $u$ solves  \eqref{me}. Thus, the Cauchy-Dirichlet  problem corresponding to  equation \eqref{me} can be uniquely solved for continuous data when $\Omega$ satisfies the above condition.

\end{proof}

\begin{rmrk}\label{R:limit}
We mention explicitly that the proof of Theorem \ref{T:existencebdd} shows that, when the initial datum $g$ is sufficiently smooth ($C^2$), then the unique solutions $u^\ve$ of the regularized problems \eqref{e:1aprox0} converge as $\ve\to 0$ to the unique viscosity solution $u$ of the Cauchy-Dirichlet problem \eqref{cd}, uniformly on compact subsets of $\Om\times [0,\infty)$, i.e., 
\[
\underset{\ve\to 0}{\lim}\ u^\ve(x,t) = u(x,t).
\] 
This fact will be used in the proof of Corollary \ref{C:commute}. 
\end{rmrk}

\section{Convergence to flow by mean curvature as  $p \to1$}\label{S:conv}

In this section we study what happens to the viscosity solutions $u_p$ of \eqref{me}, in the limit as  $p\to 1$. The next result states that such $u_p$'s converge to the unique solution of the mean curvature flow with same initial datum.

\begin{thrm}\label{T:6} Given $g\in C^2(\Rn)$, and  constant outside a compact set, for a given $p>1$, let $ u_{p} $ be a  viscosity solution  of \eqref{me} as in Theorem \ref{T:existence}. Then,  $\underset{p\to 1}{\lim}\ u_{p} = u_{0}$, where $u_{0}$ is the unique solution of the generalized mean curvature flow equation \eqref{mmc}. The convergence being uniform on every compact subset in $\Rn \times [0, \infty)$.
\end{thrm}

\begin{proof}
Take any sequence  $p_{k} \to 1$. Without loss of generality we may assume that  $p_{k} < 2$ for all $k$. Now, following \cite{ES1}, for  a given $ p_{k}$, let $ u^{\ve}_{p_{k}}$ be as above. From Section \ref{S:existence} we have 
\begin{equation}
\begin{cases}
||u^{\ve}_{p_{k}}||_{\infty} \le ||g||_{L^\infty(\Rn)},\ \ \ ||Du^{\ve}_{p_{k}}||_{\infty} \leq ||Dg||_{L^\infty(\Rn)},
\\
 ||(u^{\ve}_{p_{k}})_{t}||_{\infty}  \leq  C || D^{2} g||_{L^\infty(\Rn)},
\end{cases}
\end{equation}
where $C$ above depends on $p-1$, and thus it is uniformly bounded  for  $p < 2$. From the above estimates,  we obtain a uniform bound on the Lipschitz (in space and time) norm of $u_{p_{k}}$ for all $p_{k}$.
Therefore, we can  extract a  subsequence such that $u_{p_{m}} \to u_{0}$ locally uniformly. The function $u_0$ inherits the same uniform Lipschitz bound of the sequence. Let $\phi \in C^{2}(R^{n+1})$, and  suppose  $u_{0} -\phi$ has a  strict  local maximum at  $(x_{0},t_{0})$. Because of uniform convergence (standard  arguments as in Section \ref{S:def} above), we deduce that $u_{p_{m}} - \phi$ has a  local maximum  at  $(x_{m},t_{m})  \to (x_{0},t_{0})$.
Suppose first  $D\phi(x_0,t_0)  \neq  0$,  then $ D\phi(x_{m}, t_{m})  \neq  0$ for  large enough $m$. Then,
$\phi_{t} \leq (\delta_{ij} +  \frac{( p_{m} - 2) \phi_{i}\phi_{j}}{ |D\phi|^2})\phi_{ij}$ at $(x_{m}, t_{m})$.
Taking  the limit as $m \to \infty$, we obtain  
$\phi_{t} \leq ( \delta_{ij} -   \frac{  \phi_{i}\phi_{j}}{ |D\phi|^2})\phi_{ij}$  at   $ (x_{0},t_{0})$.
Next, if $D\phi(x_0,t_0) = 0$, then for each $m$ large enough, we can define the sequence {$a^{m}\in \Rn$} as follows
\[
a^{m} = \frac{D \phi(x_m,t_m)}{|D\phi(x_m,t_m)|},\ \ \   \text{if}\  D \phi (x_{m}, t_{m}) \neq 0.
\]  
Otherwise, when $D\phi(x_{m},t_{m}) = 0$, we know from Definition \ref{D:vsol}  that there exists $|a^{m}|\le 1$ such that
$\phi_{t} \leq ( \delta_{ij} +   ( p_{m} - 2) a^{m}_{i}a^{m}_{j})\phi_{ij}$.  
After  passing to a  subsequence, if necessary, we may assume that  $a^{m} \to a$, and  passing to the limit we obtain
$\phi_{t}  \leq  (\delta_{ij}  - a_{i}a_{j}) \phi_{ij}$  at   $(x_{0},t_{0})$.
Now,  if  we  have  local maximum and not  strict  local maximum, we  can take
$\phi_{1}(x,t) = \phi(x,t) + |x - x_{0}|^{4}  + (t-t_{0})^{4}$ and  repeat the arguments above with  $ \phi_{1}$.
Thus  $u_{0}$ is a  subsolution of \eqref{mmc}. Similarly, one shows that $u_0$ is a  supersolution. Uniqueness follows from the comparison principle. Thus, in particular, $u_{0}$ is constant outside a compact set since it coincides with the solution in \cite{ES1}. Now given any compact set $K\subset \Rn \times [0, \infty)$, for every sequence $p_{k} \to 1$, we have a subsequence $p_m$ for which there is uniform convergence of $u_{p_{m}} \to u_{0}$, and this thus implies  convergence of the whole sequence. This completes the proof.

\end{proof}

\begin{rmrk}  When $\Omega$ is a bounded domain which satisfies the geometric assumption for Theorem \ref{T:existencebdd}, as indicated above for  smooth enough Cauchy-Dirichlet datum we have uniform Lipschitz estimates. As a consequence, similarly to what was done above we conclude that a statement such as Theorem \ref{T:6} holds, i.e., $\underset{p \to 1 }{\lim}\ u_{p}=u_{0}$, where $u_{0}$ is the solution as in \cite{SZ}.
\end{rmrk}

\section{ Large-time  behavior}\label{S:ltb}

Let  $\Omega\subset \Rn$ be a smooth, convex bounded domain satisfying the hypothesis in Theorem \ref{T:existencebdd}. In this section we  only consider sufficiently smooth Cauchy-Dirichlet data, and study the large-time behavior of the corresponding solutions.

\begin{prop}\label{P:mon}
Let $u^{\ve}$ be a  solution of  the approximating problems \eqref{e:1aprox0} in $\Omega \times [0, \infty)$ corresponding to sufficiently smooth data on the parabolic boundary. Then,
\[
t \to  \int_{\Omega} ( \ve^{2} + | Du^{\ve}(x, t)|^{2})^{p/2}  dx
\]
is non-increasing.
\end{prop}

\begin{proof}
We easily find
\begin{align*}
\frac{d}{dt} \int_{\Omega } (\ve^{2} + |Du^{\ve}|^{2} )^{p/2} &  =  p \int_{\Omega } (\ve^{2} + |Du^{\ve}|^{2})^{p/2 - 1}<Du^{\ve},(Du^{\ve})_{t}>
\\
& =  -p \int_{\Omega} u^{\ve}_{t}\ \text{div}((\ve^{2} + Du^{\ve})^{p/2 - 1 } Du^{\ve}) 
\\
&  =  -p \int_{\Omega } (\ve^{2} + |Du^{\ve}|^{2} )^{p/2 -1 } (u^{\ve}_t)^{2}  \leq  0. 
\end{align*}
(the boundary  integral  vanishes  because  the boundary datum is independent of time. Here, we have made use of the  fact that, away from  $t = 0$, we have continuity of derivatives up to the boundary, a fact which follows from the classical theory, see \cite{LU}).

\end{proof}

We next prove the following result.

\begin{prop}[Energy monotonicity]\label{P:enemon}
Let $1\le p<\infty$. Then, the function $t \to \int_{\Omega}|Du(x,t)|^{p} dx$ is non-increasing.
\end{prop}

\begin{proof}
The proof for the case $p=1$ is in \cite{SZ}, thus we assume $p>1$.
Using Proposition \ref{P:mon}, the fact that $Du^{\ve}(x,t_{n}) \to  Dg$  as $ t_{n} \to 0$  for all $ x  \in  \Omega$,  and Lebesgue dominated convergence theorem, we conclude for any $t > 0$   
\[
\int_{\Omega } (\ve^{2} + | Du^{\ve}(x,t)|^{2})^{p/2} dx  \leq  \int_{\Omega} ( \ve^{2}  + |Dg(x)|^{2})^{p/2} dx.
\]
Since $u^{\ve}(\cdot, t) \to u(\cdot, t)$ weakly in $W^{1,p}(\Om)$, using lower semicontinuity and letting $ \ve \to  0$, we obtain
\begin{equation}\label{de}
\int_{\Omega} |Du(x,t)|^{p} dx \leq  \int_{\Omega} |Dg(x)|^{p} dx.
\end{equation}
For any given times $t_{1} \leq t_{2}$, we first  extend $u(\cdot,t_{1}) $ outside  $\Om$  by $ g$. Now,  let $ u_{k} = u(\cdot, t_{1}) \star \rho_{\ve_{k}}$(mollification of $u(\cdot, t_{1}$). Let $v_{k}$ be the solution of the Cauchy-Dirichlet problem in $\Omega \times [t_{1},\infty)$ with Cauchy -Dirichlet  datum $u_{k}$. So for each $k$, we obtain from \eqref{de}
\begin{equation}\label{e:est}
\int_{\Omega} |Dv_{k}(x,t_{2})|^{p} dx \leq  \int_{\Omega} |Du_{k}(x)|^{p} dx.
\end{equation}
 Since by the results in Section \ref{S:existence} we have  
 \[
 |Du_{k}|= |Du(\cdot,t_{1}) \star \rho_{\ve_k}| \leq |Du(\cdot, t_{1})| \leq C(g),  
 \]
from \eqref{e:est} we have  that  $Dv_{k}(\cdot, t_{2})$ are  uniformly bounded in $L^{p}(\Om)$.  Also, since  $u_{k} \to u(\cdot,t_{1})$ uniformly, and $Du_{k} \to Du(\cdot, t_{1})$ in $L^{p}(\Om)$, by Theorem \ref{T:max1} and (\ref{e:est}) which gives  uniform $L^{p}$ bounds we conclude that  $v_{k}(\cdot, t_{2}) \to  u(\cdot, t_{2}) $ uniformly and weakly in $W^{1,p}(\Om)$. Consequently, by using lower-semicontinuity in the left-hand side, and by taking limit in the right-hand side of \eqref{e:est}, we conclude that 
\[
\int_{\Omega} |Du(x, t_{2})|^{p} dx  \leq \int_{\Omega} |Du(x, t_{1})|^{p} dx \ \ \ \text{for all}\   t_{1} \leq  t_{2}.
\]
This gives the desired conclusion.

\end{proof}

The next result provides some interesting information on the large-time behavior of the functions  $u^{\ve}$ in Proposition \ref{P:mon}. 

\begin{thrm}\label{T:ltb}
There exists a Lipschitz continuous function $v^{\ve}\in C^{\infty}(\Omega)$ such that: 
\begin{itemize}
\item[1)] $\underset{t \to \infty}{\lim}\ u^{\ve}(x,t)=  v^{\ve}(x)$,  uniformly in  $\overline{ \Omega}$;
\item[2)] $|Dv^{\ve}(x)| \leq  C$ for every $x\in \Om$;
\item[3)] \emph{div}$((\ve^{2} + |Dv^{\ve}|^{2})^{p/2 - 1} Dv^{\ve})  =  0$;
\item[4)]  $v^{\ve} =  g$ on $\partial \Omega$.
\end{itemize}
\end{thrm}

\begin{proof}
In the following discussion the superscript $\ve$ will be omitted throughout. Following \cite{SZ}, by applying  the uniform Lipschitz  bounds, the theorem of Ascoli-Arzel\`a guarantees the existence of a  sequence ${t_k } \to \infty$, and of a Lipschitz continuous  function $v(x)$ to which $u(x, t_{k})$ converges uniformly. Now, choose a test function $\phi\in$ $C^{\infty}_{0}(\Omega)$. Then, using the fact that $u (=u^{\ve})$ is a  classical solution, integrating by parts we obtain
\begin{equation}\label{e:main}
\ \int_{t_{k}}^{t_{k+1}} \int_{\Omega}  u_{t} \phi  =  - \int_{t_{k}}^{t_{k+1}} \int_{\Omega} <S Du,D(S^{-1} \phi)> dx, 
\end{equation}
where we have let $S  = (\ve^2  + | Du|^2)^{p/2-1}$.  
By using the  mean  value theorem, we deduce that there exists $T_{k}  \in (t_{k}, t_{k+1})$ such that the absolute value of the right-hand side of \eqref{e:main} is
\[
(t_{k+1} - t_{k})\  \left|\int_{\Omega} \left[<Du(x,T_{k}),D\phi> - S^{-1}(x ,T_{k}) <Du(x,T_{k}),DS(x, T_{k})>\phi\right] dx \right| .
\]
By passing  to  a  subsequence if necessary, we may assume  that   $t_{k+1} - t_{k} \geq 1$.
We define a sequence of  functions  by letting ${u_{k}} = u(\cdot,T_{k} )$. Then, each $u_{k}$  satisfies the following divergence form equation in  $\Omega$
\[
\text{div}(( \ve^{2} + |Du_{k}|^2) ^{p/2 - 1 } Du_k) =  f_{k}(x),
\]
where   
$ f_{k}(x) = S(x, T_{k})u_{t}(x, T_{k})$.
Now, from the estimates in Section \ref{S:existence} we find
\[
|| f_{k}||_\infty \leq ||(\ve^2  + |Du_{k}|^2)^{p/2-1}||_\infty  \ ||u_t||_\infty \leq C(\ve, g),
\]
where $C(\ve, g)$ is independent of $k$, and thus the $f_k$'s have uniformly bounded $L^{2}(\Om)$ norms. So, by the structure of the principle part in the equation above, we  have by standard elliptic theory, see for  instance \cite{G}, Theorem  8.1  on p. 267, where the same proof goes through with $f(=f_{k})$ considered as a scalar term,
\[
 ||u_{k}||_{W^{2,2}(\Omega_{1})} \leq  C(\Omega_{1}), 
\]
for any $\Omega_{1}$ compactly contained in $\Omega$, $C$ being independent of $k$. 
Therefore,  by the theorem of Ascoli-Arzel\`a and standard $L^{2}$ theory, we obtain a  subsequence $ u_{k}\to v$,  uniformly  in  $\Omega$,  $Du_{k}  \to Dv$ strongly in $L^{2}$, and $ D^{2} u_{k}  \to  D^{2} v$ weakly in $\Omega_{1}$, for any $\Omega_{1}\Subset \Om$. Now, take a countable exhaustion of $\Omega$ by compact  subdomains. Thus, by employing  a  standard  diagonal process, we obtain a  sequence of times  $T_{k}$  $ \to  \infty$   such that   $ u_{k} \to  v$ uniformly and  $Du_{k} \to Dv$ pointwise  a.e. and  $D^{2}u_{k}  \to D^{2}v$,   weakly on every compact subset of $\Omega$. Also, because of uniform convergence, we have  $||Dv||_\infty \leq  C$, since this is true for each $Du_{k}$.
By Lebesgue dominated  convergence theorem, we conclude for this sequence that $Du_{k} \to Dv$ strongly in $L^{2}( \Omega)$. Let now $\phi$ be as in \eqref{e:main}, with supp $\phi$ contained in $\Omega_{1}\Subset \Om$. Because of $L^2$ convergence of gradients, we have
\begin{equation}\label{e:e1}
 \int_{\Omega} <Du_{k},D\phi> dx  \to \int_{\Omega}<Dv,D\phi> dx.
\end{equation}
Also, denoting by $S_k  = (\ve^2  + | Du_k|^2)^{p/2-1}$, and $S_v  = (\ve^2  + | Dv|^2)^{p/2-1}$, we have
\begin{align}\label{e:e2}
 \int_\Om   S_{k}^{-1} <Du_{k},DS_{k}> \phi dx  & =  (p -2) \int_\Om  \frac{ (u_{k})_{i}(u_{k})_{j}(u_{k})_{ij}}{\ve^{2} + |Du_{k}|^{2}} \phi dx
 \\
 & \underset{k\to \infty}{\longrightarrow}  (p-2) \int_\Om \frac{ v_{i}v_{j}v_{ij}}{\ve^{2} + |Dv|^{2}} \phi dx  = \int_\Om S_v^{-1}<Dv,DS_v>\phi dx.
\notag
\end{align}
(The equation \eqref{e:e2} is justified by the (strong) $L^{2}$ convergence of first derivatives, and the weak convergence of second derivatives, as stated above. More precisely, the convergence in \eqref{e:e2} is justified  by  adding and subtracting $(p-2) \int \frac{v_iv_j( u_{k})_{ij} \phi}{\ve^2+ |Dv|^2}$, and using the fact that $||(u_{k})_{ij}||_{2}$ is uniformly bounded in supp $\phi \subset \Om_1$.)

Let now $\delta > 0$. Then, the uniform convergence of the original sequence $u(\cdot,t_{k})$  implies that $u(\cdot,t_{k})$ is uniformly Cauchy. This means that for all $k$ large enough (depending on $\delta$) 
\[
|| u(\cdot, t_{k+1} ) - u(\cdot, t_{k})||_{L^{\infty}} \leq \delta.
\]
Therefore,  the absolute value of the left-hand side of  \eqref{e:main} is  
\[
\leq  \int_{\Omega} |u( x,t_{k+ 1} - u(x,t_{k})| |\phi(x)| dx  \leq  C  \delta,
\]
with the constant $C$ depending on $\phi$.
Thus, since $t_{k+1} - t_{k}\geq 1$,  we have 
\begin{equation}\label{e:e3}
\left|\int_{\Omega} \left[<D u(x, T_{k}),D\phi>  -  S^{-1}(x,T_{k})<Du(x,T_{k}),DS(x, T_{k})> \phi\right] dx \right| \leq  \frac{C \delta}{t_{k+1} - t_{k}} \leq C \delta.
\end{equation}
Therefore, from \eqref{e:e1}, \eqref{e:e2}, \eqref{e:e3} we obtain 
\[
\left|\int_{\Omega} \left[<Dv,D\phi> - S_v^{-1} <Dv,DS_v>\phi\right] dx \right| \leq C\delta, 
\]
since the same inequality is true for $u_k$, for all $k$'s sufficiently large.
From the arbitrariness of $\delta>0$, we conclude that $v$ satisfies    
\begin{equation}\label{Sv}
\int_{\Omega} S_v <Dv,D(S_v^{-1}\phi)> dx =  0,\ \ \ \ \text{for every}\ \phi\in C^{\infty}_0(\Om)
\end{equation}
Now, given any $\xi\in C^{\infty}_{0}( \Omega)$, we have that $S_v \xi\in W^{1,2}_{0}( \Omega)$. Let us choose a sequence $\phi_{k}\in C^{\infty}_{0}( \Omega)$ converging to $S_v\xi$ in $W^{1,2}(\Om)$. Without loss of generality, we can arrange that all $\phi_{k}$ be supported in some $\Om_{1}\Subset \Om$. From \eqref{Sv} we obtain
\begin{equation}\label{e:e4}
0 = \int_{\Om} S_v <Dv,D(S_v^{-1}\phi_{k})> dx = \int_\Om \left[<Dv,D\phi_{k}> -  S_v^{-1}<Dv,DS_v> \phi_{k}\right] dx .
\end{equation}
Since $\phi_{k} \to S_v\xi$ in $W^{1,2}(\Om)$, the first integral in the right-hand side of \eqref{e:e4} is easily seen to converge to $\int_\Om <Dv,D(S_v \xi)> dx$. For the second integral, since everything is supported in  $\Om_{1}$, we see from \eqref{e:e2}
\begin{align}\label{Sv2}
& \left|\int_\Om S_v^{-1}<Dv,DS_v> \phi_{k}-
S_v^{-1}<Dv,DS_v> (S_v \xi) \right| 
\\
& \leq  ||S_v^{-1}<Dv,DS_v>||_{L^2(\Om_1)} \  ||\phi_{k} - S_v\xi||_{L^2(\Om)} 
\notag\\
& = |p-2| \left\|\frac{ v_{i}v_{j}v_{ij}}{\ve^{2} + |Dv|^{2}}\right\|_{L^2(\Om_1)}  \ ||\phi_{k} - S_v\xi||_{L^2(\Om)}
\notag\\
& \leq  C  ||D^2 v||_{L^2(\Om_{1})}|| \ ||\phi_{k} - S_v\xi||_{2}  \to 0
\notag
\end{align}
Therefore, by passing to the limit for $k\to \infty$ in \eqref{e:e4}, we obtain
\[
\int_{\Omega} S_v<Dv,D\xi> dx  = 0.
\]  
By the arbitrariness of $\xi \in C^{\infty}_{0} (\Omega)$,  we conclude that  $v = v^\ve$ is a weak solution of 
\begin{equation}\label{pv}
\text{div}((\ve^{2} + |Dv|^{2})^{p/2 - 1}Dv) = 0.
\end{equation}
Since for each $k$ we have $u_{k} = g$ on $\partial \Om$, passing to the limit as $k\to \infty$ we conclude that $v = g$ on  $\partial \Om$.
Convexity of the  functional implies that  $v^{\ve}$  minimizes  $\int_\Om (\ve^{2} + |Df|^{2})^{p/2} dx$ among all $f$'s subject to boundary values $g$. The smoothness of $v^{\ve}$ follows from the elliptic  theory. We have thus proved 3) and 4). We are left with proving 1) and 2). 

Since $Du_k = Du(\cdot,T_k) \to Dv$ pointwise a.e., and the $Du_k$ satisfy uniform $L^\infty$ bounds we see that 2) holds a.e., and therefore everywhere ($v^\ve$ is smooth). Moreover, by Lebesgue dominated convergence we have $\int_{\Om} (\ve^2  + |Du_{k}|^{2})^{p/2} dx \to  \int_{\Om}( \ve^2  + |Dv|^2)^{p/2} dx$.
We now invoke Proposition \ref{P:mon} which gives the monotonicity of 
\[
t\to E_{\ve} (t) =  \int_{\Omega}(\ve^{2} + |Du^{\ve}(x,t)|^{2})^{p/2} dx.
\]
Because of \eqref{pv}, we know that $E_\ve(u_k) \overset{def}{=} E_\ve(T_k)$ decreases to $E_\ve(v)$. Given $\delta>0$ choose $N\in \mathbb N$ large enough that for $k>N$ we have $E_\ve(u_k) - E_\ve(v) \le \delta$. For $T>T_k$ and monotonicity we have
\[
0 \le E_\ve(T) - E_\ve(v) \le E_\ve(u_k) - E_\ve(v) \le \delta.
\]
This shows that $E_\ve(t)\searrow E_\ve(v) = m$, as $t\to \infty$. Now take any sequence $T_i\to \infty$. By the uniform Lipschitz bounds on the $u(\cdot,T_i)$ we know that there exists a subsequence, still denoted by the same symbol, such that $u(\cdot,T_i) \to u_0$ uniformly, and weakly in $W^{1,p}(\Om)$, as $i\to \infty$. By lower semicontinuity and convexity of the energy functional, we have
\begin{align*}
m & = \int_{\Om}( \ve^2  + |Dv|^2)^{p/2} dx = \underset{i\to \infty}{\lim} \int_{\Om}( \ve^2  + |D(\cdot,T_i)|^2)^{p/2} dx
\\
& \ge \int_{\Om}( \ve^2  + |Du_0|^2)^{p/2} dx
\end{align*}
However, since $m$ is the minimum of the energy functional, and $u_0$ and $v^\ve$ have the same boundary values, we conclude that $v^\ve = u_0$. We have thus shown that for any sequence $T_i\to \infty$
there exists a subsequence $T_{k_i}$ such that $u(\cdot,T_{k_i}) \to v$ uniformly. This establishes 1), thus completing the proof of the theorem.

\end{proof}

\begin{thrm}\label{T:minimization}
Let  $v^{\ve}$ be as in Theorem \ref{T:ltb}. Then, the  sequence $\{v^{\ve}\}_{\ve>0}$ converges uniformly to a function $v\in W^{1,p}(\Om)$, where $v$ is a $p$-harmonic function in $\Om$ having boundary values $g$. We thus have 
\[
\underset{\ve \to 0}\lim \underset{ t \to \infty}\lim\ u^{\ve}(x,t)  = \underset{\ve \to 0}\lim\ v^{\ve}(x) = v(x),
\]  
where in the first equality we have used 1) of Theorem \ref{T:ltb}.
\end{thrm}

\begin{proof}
Given the function $ v^{\ve}$ constructed in Proposition \ref{T:ltb}, take any  sequence  $\ve_{i}\searrow 0$. Because of the uniform Lipschitz bounds on the $v^\ve$, we can find a subsequence, which we continue to indicate with $v^{\ve_{i}}$, converging uniformly to some $v$, and such that $v^{\ve_{i}}$ converges weakly in $W^{1,p}$. If $f$ has boundary values $g$, by lower semicontinuity and the minimizing  property of  $v^{\ve}$, we find 
\begin{align}\label{Sv3}
\int_{\Om} | Dv|^{p} dx & \leq   \underset{ i \to \infty} \liminf\int_{\Omega} |Dv^{\ve_{i}}|^{p} dx \leq \underset{ i \to \infty} \liminf  \int_{\Omega} (\ve_{i}^{2} + |Dv^{\ve_{i}}|^{2})^{p/2} dx
\notag\\
&  \leq \underset{i \to \infty}\liminf \int_{\Omega} (\ve_{i}^{2} + |Df|^2)^{p/2} dx =  \int_{\Omega} |Df|^{p} dx
\notag
\end{align} 
Therefore,  $v$ minimizes   $\int_{\Om} |Df|^{p} dx$ with  boundary  value  $g$. Thus, every sequence $v^{\ve_{i}}$ has a converging subsequence to $v$, therefore the whole sequence converges to $v$. Since  $p$-harmonic functions are characterized by the minimizing property of the Dirichlet  integral with given   boundary values, we have proved the desired result.

\end{proof}

Theorem \ref{T:minimization} is the counterpart, for the case $p>1$ in equation \eqref{p}, of a result that in in \cite{SZ} was proved for the case $ p =1$, where it was shown that $v^{\ve}$ converge to a Lipschitz function $v$ of least  gradient. In this regard, we recall that an example in \cite{SZ} shows that, for the case  $p =1$, one has in general that 
\begin{align}
&\underset{\ve \to 0}\lim \underset{t \to \infty}\lim u^{\ve}(x, t) \neq \underset{t \to \infty}\lim \underset{\ve \to 0}\lim u^{\ve}(x, t),
\notag
\end{align}
and therefore when $p=1$ one concludes that $\underset{t \to \infty}\lim u^\ve(\cdot, t)$ might not be a function of least gradient, in general. This reveals the complexity of the large-time behavior associated with the  generalized mean curvature flow. However, we show that for $1 < p \leq 2$, the above limits do commute. We follow the ideas in \cite{ISZ}. We need the following intermediate lemma.

\begin{lemma}\label{l:com}
For $1< p\leq 2$, let $u$ be the unique viscosity solution in $\Om\times(0,\infty)$ in Theorem \ref{T:existencebdd}. Then,  
\[
\int_0^\infty \int_{\Omega} u_{t}^{2} dx dt < \infty.
\]
\end{lemma}

\begin{proof}
From the calculations in the proof of Proposition \ref{P:mon}, we obtain for any $T  > 0$
\begin{align}
& \int_{\Omega} (|Dg(x)| ^{2} +  \ve^{2})^{p/2} dx -  \int_{\Omega} (|Du^{\ve}(x,T)|^{2} + \ve^{2})^{p/2} dx 
\\
& = p \int_0^T \int_{\Omega} (\ve^{2} + |Du^{\ve}(x,t)|^{2} )^{p/2-1} u^{\ve}_{t}(x,t)^{2} dx dt.
\notag
\end{align}
We thus have for all $T > 0$
\begin{equation}\label{e:e8}
p \int_0^T \int_{\Omega} (\ve^{2} + |Du^{\ve}|^{2} )^{p/2-1} (u^{\ve}_{t})^{2} dx dt \leq \int_{\Omega} (|Dg| ^{2} +  \ve^{2})^{p/2} dx\le C(g),
\end{equation}
where $C(g)$ is independent of $T$ and $\ve$. On the other hand, we know there is a constant $B = B(g)> 0$, independent of $T$ and $\ve$, such that $|Du^\ve(x,t)|\le B$ for every $x\in \Om$ and every $0\le t\le T$. We conclude that, when $1<p \leq 2$, then  
\[
(\ve^{2} + |Du^{\ve}(x,t)|^{2})^{p/2-1} \ge B > 0,\ \ \ \ \text{on}\ \Om\times (0,T).
\]
Therefore, for any $T>0$, we have
\[
\int_0^T \int_{\Omega}(u^{\ve}_{t})^{2} dx dt \le C^\star,
\]
where $C^\star$ is independent of $T, \ve$. Since $u^{\ve}_{t} \to u _{t}$ weakly on compact sets, we have reached the desired conclusion.
 
\end{proof}

We now state the following theorem.

\begin{thrm}\label{T:conv}
For $1< p\leq 2$, let $u$ be the unique viscosity solution in $\Om\times(0,\infty)$ in Theorem \ref{T:existencebdd}. Then, as  $ t \to \infty$ the function $u$ converges to the unique $p$-harmonic  function $v$ in $\Om$ having boundary values $g$ on $\p \Om$.
\end{thrm}

\begin{proof}
Let $C_{T}$ be the cylinder  $ \Omega  \times [0, T]$. Consider the sequence  $u^{k}$  defined  by $u^{k}(x,t) = u(x,t+k)$. Because of the uniform Lipschitz bounds in Section \ref{S:existence}, by the theorem of Ascoli-Arzel\`a we  have a subsequence  $u^{k} \to  v$ locally  uniformly in $\overline \Om \times [0,\infty)$, and thus  uniformly in the compact set  $\overline{C_{T}}$. It is thus easily verified that the function $v$ is also a solution of \eqref{p}. We claim that  $v$ is independent of $t$. From Lemma \ref{l:com} we find that 
\begin{equation}
\underset{k \to \infty}\lim \int_{C_{T}} (u^{k}_{t})^{2} dx dt  = \underset{k \to \infty}\lim  \int_k^{k + T} \int_{\Om} u_{t}^2 dx dt \leq \underset{k \to \infty}\lim \int_k^\infty \int_{\Om} u_{t}^2 dx dt =  0. 
\end{equation}
Since $u^{k}_{t} \rightharpoonup v_{t}$ in $\overline{C}_{T}$, by lower semicontinuity we obtain 
\[
 \int_{C_{T}} v_{t}^{2} dx dt = 0.
 \]
As in the proof of Theorem 4.4 in \cite{ISZ}, we conclude that  $v$ is independent of $t$ in $C_{T}$. By the arbitrariness of $T$, we conclude that $v$ is independent of $t$ in $\Om \times [0, \infty)$.  The fact that  $v$ is $p$-harmonic is seen from  Lemma \ref{l:pevol}. It remains to be seen  that $u(\cdot,t)$ converges to $v$ uniformly as $t\to \infty$. Given $\ve > 0$, choose $k$ large enough such that $|u^{k}(x,0) - v(x)| \leq \ve$ (because of uniform convergence in the compact set $\overline{\Om} \times \{0\}$). Now, we apply Theorem \ref{T:max1}  to conclude that  $|u^{k}(x,t) - v(x)|  = |u(x,t+ k) - v(x)| \leq \ve$ for all $x\in \overline \Om$, and every $t \geq k$. (note that on the lateral boundary both functions equal $g$). This concludes the proof.

\end{proof}

\begin{cor}\label{C:commute}
For $1< p\leq 2$, let $u$ be the unique viscosity solution in $\Om\times(0,\infty)$ in Theorem \ref{T:existencebdd}. Then, 
\begin{align}
&\underset{\ve \to 0}\lim \underset{t \to \infty}\lim u^{\ve}(x,t) = \underset{t \to \infty}\lim \underset{\ve \to 0}\lim u^{\ve}(x,t),
\notag
\end{align}
\end{cor}

\begin{proof}
From Remark \ref{R:limit} we have
\[
\underset{\ve\to 0}{\lim}\ u^\ve(x,t) = u(x,t),
\] 
the convergence being uniform on compact subsets of $\Om \times [0,\infty)$. Using this fact, and Theorem \ref{T:conv}, we conclude
\[
\underset{t \to \infty}\lim \underset{\ve \to 0}\lim u^{\ve}(x,t) = \underset{t \to \infty}\lim\ u(x,t) = v(x).
\]
This fact, combined with Theorem \ref{T:minimization}
gives the desired conclusion.

\end{proof}

\begin{rmrk}\label{R:open}
Corollary \ref{C:commute} makes the case $1<p\le 2$ of equation \eqref{p} very different from that of \eqref{mmc}, when $p=1$. For \eqref{mmc} there also exists an equilibrium solution independent of time, but the conclusion of Corollary \ref{C:commute} does not hold in general, see \cite{SZ} and \cite{ISZ}. It remains an interesting open question whether Corollary \ref{C:commute} continues to be valid for $p > 2$. In such case one needs to find an appropriate replacement of Lemma \ref{l:com}.
\end{rmrk}

\section{ Energy  estimates  and monotonicity}\label{S:eemon}

For the case $ p =1$, the  following monotonicity of the energy of the unique (bounded) solution $u$ to the Cauchy problem for \eqref{mmc} was established in \cite{ES3}: 
\begin{equation}\label{fm}
\int_{\Rn} |Du|(x,t_{2}) dx  \leq  \int_{\Rn} |Du|(x, t_{1}) dx   \quad  for\quad  all\   t_{1} \leq t_{2}.
\end{equation}
For $p>1$ we can prove an analogous monotonicity result.

\begin{thrm}[Unweighted energy monotonicity]\label{T:1mon}
 Let $u$ be the unique viscosity solution of \eqref{me} obtained in Theorem \ref{T:existence}, where the initial datum $g$ is Lipschitz continuous, and constant outside a compact set. Then, 
\begin{equation}\label{fm1}
\int_{\Rn} |Du|^p(x, t_{2}) dx  \leq  \int_{\Rn} |Du|^p(x, t_{1}) dx   \quad  \text{for all}\   t_{1} \leq t_{2}.
\end{equation}
\end{thrm}

\begin{proof}
By subtracting a constant, we can without loss of generality assume that $g$ be compactly supported. Denote by $u^\ve$ the solution to the regularized Cauchy problem \eqref{e:1aprox}. First, we also assume that $g$  is smooth, a fact which ensures bounds on derivatives of $u^{\ve}$, as in Section \ref{S:existence}. In this first part of the proof we adapt a beautiful argument in the proof of Lemma 2.1 in \cite{ES3}. Letting $\phi(x)=  e^{-b(1+|x|^2)^{1/2}}$, we define 
\[
F^{\ve}_{b}(t) = \int_{\Rn} \phi(x)^{2}( |Du^{\ve}(x,t)|^2 + \ve^{2})^{p/2} dx,
\]
and note that
\[
F^{\ve}_{b}(0) = \int_{\Rn} \phi(x)^{2}( |Dg(x)|^2 + \ve^{2})^{p/2} dx.
\]
Differentiating gives   
\begin{align*}
(F^{\ve}_{b})^{'}(t)  = & \ p \int_{\Rn} \phi^{2} ( |Du^{\ve}|^2 + (\ve)^{2})^{p/2 -1}<Du^{\ve},Du^{\ve}_{t}> dx. 
\\
 =  & -p  \int_{\Rn} \phi^{2} \text{div}(|Du^{\ve}|^2  + (\ve)^{2})^{p/2 -1}Du^{\ve})u^{\ve}_{t } dx
\\
&  - 2p \int_{\Rn}\phi<D\phi,Du^{\ve}>(|Du^{\ve}|^2  + \ve^2)^{p/2-1}u^{\ve}_{t} dx.
\end{align*}
If we now let  
\[
H^{\ve}  =  \text{div}(|Du^{\ve}|^2 +(\ve)^{2})^{p/2 -1} Du^{\ve}),
\]
then we can write the equation as
\[
H^{\ve} (|Du^{\ve}|^2  + \ve^2)^{1-p/2} = u^{\ve}_{t}.
\]
We thus find
\[
(F^{\ve}_{b})^{'}(t) = - p \int_{\Rn} \phi^2 (H^{\ve})^{2} (|Du^{\ve}|^2  + \ve^2)^{1-p/2} dx - 2p \int_{\Rn} \phi <D\phi,Du^{\ve}> H^{\ve} dx.
\]
Now, we have trivially $|Du^{\ve}|  \leq (|Du^{\ve}|^2  +  \ve^2)^{1/2}$. If
we write
\[
(|Du^{\ve}|^2  + \ve^2) = (|Du^{\ve}|^2  + \ve^2)^{1- p/2}(|Du^{\ve}|^2 + \ve^2)^{p/2},
\]
then, by Cauchy-Schwarz inequality, we easily obtain
\[
(F^{\ve}_{b})^{'}(t)  \leq p  \int_{\Rn} |D\phi|^2 (|Du^{\ve}|^2  + \ve^2)^{p/2} dx \leq b^2 p F^{\ve}_{b}(t),
\]
where we have used $| D\phi|\leq  b|\phi|$. Gronwall's inequality now easily gives for every $t\ge 0$  
\begin{equation}\label{e:ap}
F^{\ve}_{b}(t)  \leq  e^{b^{2}p t}  \int_{\Rn} \phi^{2}(x) (|Dg(x)|^2 + \ve^{2})^{p/2} dx.
\end{equation}  
(The inequality \eqref{e:ap} can be justified by taking a  sequence $t_{j} \searrow 0$, with $ t_{j} < t $, and noting that $Du^{\ve}(\cdot,t_{j}) \to Dg$,    $\phi^2 |Du^{\ve}|^{p}\leq \phi^2 ||Dg||_{\infty}^{p}$ which is in $L^{1}(\Rn)$ and then using Lebesgue dominated convergence theorem).
Let $K\subset \Rn$ be an arbitrarily fixed compact set. From \eqref{e:ap} we thus obtain for every fixed $t> 0$
\[
\int_{K} \phi(x)^{2}( |Du^{\ve}(x,t)|^2 + \ve^{2})^{p/2} dx \le  e^{b^{2}p t}  \int_{\Rn} \phi^{2}(x) (|Dg(x)|^2 + \ve^{2})^{p/2} dx.
\]
With $t> 0$ fixed, select a sequence $\ve_j\searrow 0$ such that 
$u^{\ve_j}(\cdot,t) \to u(\cdot,t)$  weakly in $W^{1,p}_{loc}(\Rn)$. 
Letting $j\to \infty$ in the latter inequality, by using the lower semicontinuity in the left-hand side and Lebesgue dominated convergence in the right-hand side, we find
\begin{equation}\label{e:ap1}
\ \int_{K} \phi^2(x) |Du(x, t)|^{p} dx  \leq  \int_{\Rn} \phi^2(x) |Dg(x)|^{p} dx. 
\end{equation}
Letting  $b\to 0$ in \eqref{e:ap1}, we obtain
\[
\int_{K} |Du(x, t)|^{p} dx  \leq  \int_{\Rn} |Dg(x)|^{p} dx   \quad  \ \ \text{for all}\   t  \in  [0, \infty).
\]
Since the latter estimate is true for all compact $K\subset \Rn$, by the monotone convergence theorem we conclude
\begin{equation}\label{e:ap3}
\ \int_{\Rn} |Du(x, t)|^{p} dx  \leq  \int_{\Rn} |Dg(x)|^{p} dx   \quad  \ \ \text{for all}\   t  \in  [0, \infty).
\end{equation}
To extend the estimate \eqref{e:ap3} to the case when $g$ is Lipschitz, we consider the $\ve_{k}$-mollifications of $g$ and call them $g_{k}$. Then, $g_{k} \to g $ uniformly  and in $W^{1,p}(\Rn)$. Let  $u_{k}$ be the solution to the Cauchy problem with initial datum $g_{k}$.  From uniform Lipschitz bounds in Section \ref{S:existence}, and Theorem \ref{T:max4}, we have that, at  any given time $t> 0$,  $u_{k}(\cdot,t) \to u(\cdot,t)$ uniformly, and  weakly in $W^{1,p}_{loc}(\Rn)$. Since \eqref{e:ap3}) holds for  $u_{k}$ and $g_{k}$, we first bound from below, as before, the integral in the left-hand side over a compact set $K$, use lower semicontinuity in the left-hand side, Lebesgue dominated convergence in the right-hand side, and finally let $K  \nearrow \Rn$ to conclude that \eqref{e:ap3} continues to hold when the initial datum $g$ is Lipschitz  continuous.

Finally, we establish \eqref{fm1}. With this objective in mind, let $G_{p}$ be the notable solution in Proposition \ref{P:es} above, and set $V(x,t) = G_{p}(x,t+1)$. We first claim that, for a  given $t\ge 0$, we have
\begin{equation}\label{ubV}
|u(x,t)| \leq C(g)  V(x,t),\ \ \ \ \ x\in \Rn.
\end{equation}
In order to prove \eqref{ubV}, we observe that if $g$ is Lipschitz continuous and compactly supported, then there exists a constant $C = C(g)\ge 0$ such that 
\[
- C V(x,0) \le g(x) \leq C V(x,0).
\]
Theorem \ref{T:max3}  then guarantees that \eqref{ubV} be true. We explicitly note that \eqref{ubV} implies, in particular, that $\lim_{|x| \to \infty}u(x, t) = 0$. And that, furthermore,
\[
\int_{\Rn} |Du(x,t)|^p dx <\infty,\ \ \ \ \ \text{for every}\ t\ge 0.
\]
However, this latter fact is already implied by the quantitatively precise estimate \eqref{e:ap3}.
Now, given $t_{1} \leq t_{2}$, for each $k\in \mathbb N$ let $h_{k}\in C^\infty_0(\Rn)$, $0\le h_k\le 1$, with $h_{k} =1$ for $|x| \leq k$ and $h_k\equiv 0$ for  $|x| \geq 2k$. Set  $g_{k} = h_{k} u(\cdot, t_{1})$ and let $u_{k}$ denote the solution to the Cauchy problem in $\Rn\times [t_{1}, \infty)$, corresponding to initial datum $g_{k}$. Because of \eqref{ubV}, it is easy to recognize that 
\[
g_{k} \to u(\cdot, t_{1}),\ \ \ \ \text{uniformly in}\ \Rn.
\]
Now, using the $L^{\infty}$ bounds of the solutions and their gradients from Section \ref{S:existence}, and the fact that $||h_{k}||_\infty \leq 1$, $||Dh_{k}||_\infty \leq C/k \leq C$, we obtain 
\[
||Du_{k}(\cdot, t_{2})||_\infty = ||Dg_{k}||_\infty \leq  (||Dh_{k}||_\infty || u(\cdot,t_1)||_\infty + ||h_{k}||_\infty ||Du(\cdot,t_1)||_\infty)\leq C(g). 
\]
Thus, by Theorem \ref{T:max4} and uniform Lipschitz bounds, we conclude that   $u_{k}(\cdot, t_{2}) \to u (, t_{2})$ uniformly, and  weakly in $W^{1,p}_{loc}(\Rn)$. Now,
\[
 \int_{\Rn} |Dg_{k}(x)|^{p} dx =  \int_{|x| < k}  |Du(x, t_{1})|^{p} dx + \int_{k < |x| < 2k} |Dg_{k}(x)|^{p} dx.
\]
By monotone convergence, the first integral in the right-hand side converges to $\int_{\Rn} |Du(x, t_{1})|^{p} dx$. We claim that 
\[
\underset{k\to \infty}{\lim} \int_{k < |x| < 2k} |Dg_{k}(x)|^{p} dx = 0.
\]
To recognize this fact, we observe that, using the estimate $|Dh_{k}| \leq  c/k$, we see that the integral is estimated from above by
\[
C\left(\int_{|x| > k} |Du(x, t_{1})|^{p} dx + \frac{1}{k^{p}} \int_{|x|> k} |u(x, t_{1})|^{p} dx\right).
 \]
From \eqref{e:ap3} the first  integral converges to $0$ as $k \to \infty$. The second integral, instead, also converges to $0$ because of \eqref{ubV}. We conclude that 
\[
\int |Dg_{k}(x)|^{p} dx  \to  \int |Du(x, t_{1})|^{p} dx.
\]
On the other hand, the energy estimate \eqref{e:ap3} allows to conclude that
\[
\int_{\Rn} |Du_k(x, t_2)|^{p} dx  \leq  \int_{\Rn} |Dg_k(x)|^{p} dx.
\]
At this point, we can repeat the argument following \eqref{e:ap3}, and passing to the limit as $k\to \infty$, 
we reach the desired conclusion \eqref{fm}.

\end{proof}

Our next objective is establishing a monotonicity result subtler than Theorem \ref{T:1mon}. To provide the reader with some motivation and historical background we recall that in \cite{S} Struwe proved the a regular solution $u$ of the harmonic map flow from $\R^m\times (0,T_0)$ into a compact manifold $\mathcal N$ satisfies the following monotonicity result: \emph{for any $x\in \R^m$ and any fixed $0<T<T_0$, the function
\[
t \to E(t)=(T-t)\int_{\Rn} |\nabla u(y,t)|^2 G(x,y,T-t) dy,\ \ \ \ 0<t<T,
\]
is non-increasing}. Here, we have denoted by $G(x,y,t) = G(x-y,t) = (4\pi t)^{-\frac{n}{2}} \exp(-\frac{|x-y|^2}{4t})$ the standard heat (or Gauss-Weierstrass) kernel. As it is well-known, see for instance \cite{S} and the references therein, this monotonicity theorem plays an important role in the study of the regularity of solutions of the harmonic map flow.

Struwe's result, specialized to the case when $\mathcal N = \R$, is concerned with the linear case $p=2$. In what follows, we will prove that, quite remarkably, when $p\not= 2$, viscosity solutions of the nonlinear singular equation \eqref{me} 
satisfy a monotonicity theorem similar to Struwe's result. We have in fact the following theorem.

\begin{thrm}[Generalized Struwe's monotonicity formula]\label{T:struwe}
Let $u$ be the unique viscosity solution of \eqref{me} obtained from Theorem \ref{T:existence} when the initial datum  $g$ is  Lipschitz continuous and constant outside a compact set. For every $x\in \Rn$ and $T>0$ the function   
\[
t \to E(t) =  (T-t)^{p/2} \int_{\Rn}  |Du(y,t)| ^{p} G(x,y,T-t) dy
\]
is non-increasing on the interval $0 \leq  t \leq  T$.
\end{thrm}

Before proving Theorem \ref{T:struwe} we need to establish the following intermediate result which asserts the monotonicity  of the weighted  energy  of the  approximations. In the statement of the next result, by a \emph{regular solution} of the Cauchy problems \eqref{e:1aprox} we intend a bounded solution having bounded partial derivatives up to order three. We note explicitly that, when the initial datum $g$ is sufficiently smooth, the regular solutions defined constructed in Section \ref{S:existence} amply satisfy such requirement. Therefore, they coincide with the regular solutions in the sense of \cite{S}.

\begin{thrm}\label{T:monoreg}
Let $ u^{\ve}$ be a  regular  solution of \eqref{e:1aprox}. Then, for any $x\in \Rn$ and $T>0$, the function
\[
t \to E_{\ve}(t) =(T-t)^{\frac{p}{2}} \int_{\Rn} (|Du^{\ve}(y,t)|^2 +\ve^2)^{\frac p2} G(x,y,T-t) dy,
\]
is non-increasing on the interval $0< t \le  T$.
\end{thrm}

\begin{proof}
Consider $ u^{\ve}$ as above, and for $\sigma\in \Rn$, let $\Phi_{\ve}(\sigma)= \frac{2}{p} (\ve^{2} + |\sigma|^2)^{p/2}$. 
In the following considerations, all the  $\ve$ super- and subscripts will be routinely omitted. Thus, by rewriting the equation \eqref{e:1aprox}, we see that $u (= u^\ve)$ satisfies
\begin{equation}\label{1}
\text{div}(\Phi'(|Du|^2)Du)=\Phi'(|Du|^2)u_t,   \quad\quad\quad \text{in} \quad \Rn  \times [0,\infty).
\end{equation}
Also, let us set for brevity $\rho(y,t) = G(x,y,T-t)$, and notice that $\rho$ satisfies the backward heat equation in $\Rn \times(-\infty,T)$, i.e.,
\begin{equation}\label{bh}
\Delta \rho +\rho_t = 0.
\end{equation}
For the sake of convenience, we will continue to denote by $E(t)$ the energy $E_\ve(t)$ in the statement of Theorem \ref{T:monoreg}, multiplied by a factor of $\frac 1p$, i.e., 
\[
E(t) = (T-t)^{\frac{p}{2}} \int_{\Rn} \Phi(|Du(y,t)|^2) \frac{\rho(y,t)}{2} dy.
\]
Differentiating, we find
\begin{equation}\label{2}
E'(t)=(T-t)^{\frac{p}{2}}\int_{\Rn} \left[\frac{-p}{2(T-t)} \frac{\rho}{2} \Phi + \frac{\rho_t}{2} \Phi + \frac{\rho}{2}(\Phi(|Du|^2))_t \right] dy.
\end{equation}
Now
\begin{equation}
(\Phi(|Du|^2))_t = 2(\Phi'u_t)_i u_i -4\Phi'' u_{ij} u_j u_i u_t.
\end{equation}
Using  the equation \eqref{1} we find
\begin{equation}\label{e:1}
(\Phi(|Du|^2))_t = 2(\Phi'u_t)_i u_i -2u_t [\Phi' u_t - \Phi' u_t \Delta u].
\end{equation}
Replacing \eqref{e:1} into \eqref{2}, and using \eqref{bh}, gives
\begin{equation}\label{e:2}
E'(t)=(T-t)^{\frac{p}{2}}\int_{\Rn} \left[\frac{-p}{2(T-t)} \frac{\rho}{2} \Phi - \frac{\Delta
\rho}{2} \Phi + \rho u_i (\Phi' u_t)_i - \rho \Phi' u^2_t + \rho \Phi' u_t \Delta u \right] dy.
\end{equation}
We now integrate by parts the term 
\begin{equation*}
\int_{\Rn} \rho u_i (\Phi' u_t)_i dy = -\int_{\Rn} \rho_i u_i \Phi' u_t dy - \int_{\Rn} \rho \Delta u u_t \Phi' dy .
\end{equation*} 
Substitution in (\ref{e:2}) gives 
\begin{align}
 E'(t)= & (T-t)^{\frac{p}{2}}\int_{\Rn} \left[- \frac{p}{2(T-t)} \frac{\rho}{2} \Phi - \frac{\Delta
\rho}{2} \Phi - \rho \Phi' u_t \Delta u -\Phi' u_t <Du,D\rho> - \rho \Phi' u^2_t + \rho \Phi' u_t \Delta u \right] dy
\\
& = (T-t)^{\frac{p}{2}}\int_{\Rn} -\frac{\rho}{\Phi'} \left[\Phi' u_t + <Du,\frac{D\rho}{\rho}> \Phi' \right]^2 dy 
\notag\\
& + (T-t)^{\frac{p}{2}}\int_{\Rn} \left[\rho \Phi' (<Du,\frac{D\rho}{\rho}>)^2 + \Phi' u_t <Du, D\rho> - \frac{\Delta \rho}{2} \Phi - \frac{p}{2(T-t)} \frac{\rho}{2} \Phi \right ] dy.
\notag
\end{align}

We have thus proved
\begin{equation}
 E'(t)=  (T-t)^{\frac{p}{2}}\int_{\Rn} -\rho \Phi' \left[u_t + <Du,\frac{D\rho}{\rho}> \right]^2 dy  +  G(t)
\end{equation}
with
\begin{equation}\label{G}
G(t) = (T-t)^{\frac{p}{2}}\int_{\Rn} \left[\rho \Phi' (<Du,\frac{D\rho}{\rho}>)^2 + \Phi' u_t <Du, D\rho> - \frac{\Delta \rho}{2} \Phi - \frac{p}{2(T-t)} \frac{\rho}{2} \Phi \right ] dy.
\end{equation}

In order to proceed we establish the following
\begin{lemma}\label{L:G} 
The function $G$ defined by the equation \eqref{G} is given by 
\[
G(t)= (T-t)^{\frac{p}{2}}\int_{\Rn} \frac{\rho}{2(T-t)}\left[\Phi'(|Du|^2)|Du|^2 - \frac{p}{2} \Phi(|Du|^2) \right] dy.
\]
\end{lemma}

\begin{proof}
From the  equation \eqref{1} we have
\begin{align}\label{3}
& \int_{\Rn} \Phi' u_t <Du,D\rho> dy = \int_{\Rn} <Du,D\rho> \text{div}(\Phi' Du) dy
\\
& = \int_{\Rn} \left[ -\Phi' u_{ij} u_i \rho_j\ -\ \Phi' \rho_{ij} u_i u_j\right] dy
\notag\\
& = \int_{\Rn} \left[-<D(\frac{\Phi(|Du|^2)}{2}, D\rho> - \Phi' <D^2 \rho(Du),Du>\right] dy
\notag\\
& = \int_{\Rn}\ \left[ \frac{\Phi}{2} \Delta \rho\  -\ \Phi'<D^2 \rho(Du),Du>\right] dy,
\notag
\end{align}
where we have denoted by $D^2\rho$ the Hessian matrix of $\rho$.
By substituting \eqref{3} in \eqref{G}, we obtain
\begin{equation}\label{G2}
G(t) = (T-t)^{\frac{p}{2}}\int_{\Rn} \left[\rho \Phi' (<Du,\frac{D\rho}{\rho}>)^2 + \Phi \frac{\Delta \rho}{2} - \rho \Phi'<\frac{D^2 \rho}{\rho}(Du),Du> - \Phi \frac{\Delta \rho}{2} - \frac{p}{2(T-t)}{\rho}{2} \Phi \right] dy.
\end{equation}
We now notice that
\[
\rho = (4\pi(T-t))^{-\frac{n}{2}} f\left(-\frac{r^2}{4(T-t)}\right)
\]
with $f(s)=e^s$, and $r = |y-x|$. One has
\[
\rho_i = (4\pi(T-t))^{-\frac{n}{2}} f' \left(-\frac{r^2}{4(T-t)}\right)\left(-\frac{y_i -x_i}{2(T-t)}\right).
\]
Since 
\[
f'(s)=f''(s)=f(s),
\]
we obtain
\[
\rho_i = \left(-\frac{y_i -x_i}{2(T-t)}\right)\rho ,
\]
\[
\rho_{ij} = - \delta_{ij} \frac{\rho}{2(T-t)} + \frac{(y_i-x_i)(y_j-x_j)}{4(T-t)^2} \rho .
\]
In conclusion, we have 
\begin{align}\label{5}
& \frac{D\rho}{\rho} = -\frac{y-x}{2(T-t)}
\\
& \frac{D_{ij} \rho}{\rho} = - \delta_{ij} \frac{1}{2(T-t)} + \frac{(y_i-x_i)(y_j-x_j)}{4(T-t)^2}.
\notag
\end{align}

Using \eqref{5} in \eqref{G2}, we obtain
\begin{align*}
G(t) = & (T-t)^{\frac{p}{2}}\int_{\Rn} \left[\rho \Phi' \frac{<Du,y-x>^2}{4(T-t)^2 }+ \rho \Phi' \frac{|Du|^2}{2(T-t)} - \rho \Phi'\frac{<Du,y-x>^2}{4(T-t)^2} -  \frac{p}{2(T-t)}{\rho}{2} \Phi \right] dy
\notag\\
& = (T-t)^{\frac{p}{2}}\int_{\Rn} \frac{\rho}{2(T-t)}\left[\Phi'(|Du|^2)|Du|^2 - \frac{p}{2} \Phi(|Du|^2) \right] dy, 
\end{align*}
which gives the desired conclusion.

\end{proof} 

With Lemma \ref{L:G} in hands we now resume the proof of Theorem \ref{T:monoreg}. Substituting in \eqref{G} above the  explicit form of the function $\Phi(\sigma) = \frac{2}{p} (\ve^{2} + |\sigma|^2)^{p/2}$, we obtain 
\begin{equation*}
G(t) =  (T-t)^{p/2} \int_{\Rn} \frac{\rho}{2(T-t)}[( \ve^2 + |Du^{\ve}|^2)^{p/2 - 1}|Du^{\ve}|^2  -  ( \ve^2  + |Du^{\ve}|^{2} )^{p/2}]
\end{equation*} 
Therefore,    
\begin{equation*}
G_{\ve}(t) =  (T-t)^{p/2} \int_{\Rn} \frac{\rho}{2(T-t)}[ (\ve^2 + |Du^{\ve}|^2)^{p/2 - 1}(-\ve^2)] \leq  0.
\end{equation*}
This shows that $ E_{\ve}'(t) \leq 0$, thus completing the proof of the theorem.

\end{proof}



We can now turn to the

\begin{proof}[Proof of Theorem \ref{T:struwe}]
By subtracting a constant, we can assume without loss of generality that $g$ be compactly supported. In a first step, we  also assume that  $g$ be smooth. But then, Theorem \ref{T:monoreg} gives for the corresponding $u^{\ve}$
\begin{equation}\label{e:3}
E_{\ve}(t_{2} ) \leq  E_{\ve}(t_{1}) \ \ \ \quad    t_{2} \geq t_{1}.
\end{equation}
Moreover, since for any compact set $K\subset \Rn$ we trivially have
\begin{equation*}
E_{\ve}(t) \geq  (T-t)^{p/2} \int_{K} |Du^{\ve}(y,t)|^{p} G(x,y,T-t) dy,
\end{equation*}
we obtain from \eqref{e:3}
\[
(T-t_2)^{p/2} \int_{K} |Du^{\ve}(y,t_2)|^{p} G(x,y,T-t_2) dy \le (T)^{\frac{p}{2}} \int_{\Rn} (|Dg(y)|^2 +\ve^2)^{\frac p2} G(x,y,T) dy.
\]
Here, we have made use of the fact that for a sequence $t_j\searrow 0$, with $t_j < t_2$ for every $j\in \mathbb N$, we have $Du^{\ve} (\cdot, t_{j}) \to Dg$ as  $j\to \infty$. We note that $|Du^{\ve}(\cdot, t_{j})|^{p} G(x,\cdot, T - t_{j}) \leq |Du^{\ve}(\cdot, t_{j})|^{p} G(x,\cdot, T) \leq ||Dg||^{p}_{\infty} G(x,\cdot,T )$, which belongs to $L^1(\Rn)$, and thus we can use Lebesgue dominated convergence theorem.

Now, because of the uniform bound of the solutions and their gradients in terms of $g$, there exists a subsequence $\ve_j\searrow 0$, such that $u^{\ve_j}(\cdot,t_i) \rightharpoonup u(\cdot,t_i)$ in $W^{1, p}_{loc}(\Rn)$, for $i=1,2$. Therefore, letting $\ve_j \to 0$, and using lower semicontinuity in the left-hand side of the latter inequality, and Lebesgue dominated convergence theorem in the right-hand side (which we can use since we are integrating against a Gaussian measure on $\Rn$), we obtain
\[
(T-t_2)^{p/2} \int_{K} |Du(y,t_2)|^{p} G(x,y,T-t_2) dy \le (T  )^{\frac{p}{2}} \int_{\Rn} |Dg(y)|^p G(x,y,T) dy.
\]
Letting  $t_{2}= t$ and $K_j\nearrow \Rn$, we conclude that for every $t\ge 0$ the following energy decay estimate holds
\begin{equation}\label{e:ap5}
(T-t)^{p/2} \int_{\Rn} |Du(y,t)|^{p} G(x,y,T-t_2) dy \leq (T)^{p/2} \int_{\Rn} |Dg(y)|^{p} G(x,y,T) dy. 
 \end{equation}
When $g$ is Lipschitz continuous and compactly supported, let $g_{k}$ denote the $\ve_{k}$ mollification, and let $u_{k}$ be the corresponding  solutions with initial datum $g_{k}$. Then, by Theorem \ref{T:max4} and the uniform Lipschitz bounds in Section \ref{S:existence}, we have $ u_{k}(\cdot,t) \to u(\cdot,t)$ uniformly and weakly in $W^{1,p}_{loc}(\Rn)$. Therefore, since the estimate \eqref{e:ap5} holds for each $u_k, g_k$, repeating the limiting arguments which have already been first used several times, 
we conclude that the energy estimate \eqref{e:ap5} continues to be valid for Lipschitz $g$. 

At this point, using the crucial estimate \eqref{ubV}, we can complete the proof of the monotonicity of the weighted energy by arguing as in the proof  of Theorem \ref{T:1mon}. We leave the details to the reader.


\end{proof}

For the case of the motion by mean curvature equation \eqref{mmc}, the comparison with the function $V$ as in \eqref{ubV} does not work. However, we already know that the solutions obtained in \cite{ES1} are constant outside a compact set. Thus, the intermediate step of multiplying them by the cutoff function $h_{k}$ is not required as above. This  allows us to assert an energy decay monotonicity in the case $p=1$. The calculations are justified by arguments similar to those presented above in the case $p>1$, but using the bounds in \cite{ES1}, page 655. We omit the relevant details.

\begin{thrm}[Weighted monotonicity for $p=1$]
Let $u$ be the unique viscosity solution of \eqref{mmc} with an initial datum $g$ Lipschitz continuous and constant outside a compact set. For every $x\in \Rn$ and $T> 0$ the function   
\[
t \to E(t) =  (T-t)^{1/2} \int_{\Rn}  |Du(y,t)| G(x,y,T-t) dy
\]
is nonincreasing on the interval $0 \le  t \leq  T$.  
\end{thrm}

Finally, we close this paper with a corollary of Theorem \ref{T:struwe} which generalizes to the  nonlinear singular equation \eqref{me} Struwe's monotonicity formula for the case $p=2$, see Lemma 3.2 in \cite{S}.
 
\begin{cor}\label{C:struwe}
Let $u$ be a viscosity solution as in Theorem \ref{T:struwe}. Then, the function
\[
 I(r)= E(T-r^2) = r^{p} \int_{\{t=T-r^2\}} |Du(y,t)|^{p} G(x,y,T-t) dy,
 \]  
is nondecreasing for any $0<r\le  \sqrt{T}$. 
\end{cor}

\begin{rmrk}
The  energy estimates  and monotonicity cannot be expected to hold for a  solution of \eqref{me} without any growth assumption since, even for the heat equation, Tychonoff's solution violates it. In our case, all solutions are  bounded, as seen in the existence theorems. 
\end{rmrk}

\end{document}